\documentclass[a4paper,11pt,reqno]{amsart}
\usepackage[a4paper,hscale=0.7,vscale=0.7,centering]{geometry}
\usepackage{color}

\newcommand{\erre}{\mathbb{R}}
\newcommand{\E}{\mathbb{E}}
\newcommand{\EE}{\mathcal{E}}
\renewcommand{\P}{\mathbb{P}}
\newcommand{\m}{\bar{\mu}}

\newcommand{\ip}[2]{\langle #1,#2 \rangle}
\newcommand{\bip}[2]{\left\langle #1,#2 \right\rangle}

\newsymbol\lesssim 132E

\newtheorem{prop}{Proposition}
\newtheorem{thm}[prop]{Theorem}

\newtheorem{lemma}[prop]{Lemma}
\newtheorem{defi}[prop]{Definition}
\theoremstyle{remark}
\newtheorem{rmk}[prop]{Remark}

\begin{document}

\title[Reaction-diffusion equations with jumps]{Well-posedness and
  asymptotic behavior for stochastic reaction-diffusion equations with
  multiplicative Poisson noise}

\author{Carlo Marinelli}

\address{Facolt\`a di Economia, Universit\`a di Bolzano, Piazza
  Universit\`a 1, I-39100 Bolzano, Italy.}
\urladdr{http://www.uni-bonn.de/$\sim$cm788}

\author{Michael R\"ockner}

\address{Fakult\"at f\"ur Mathematik, Universit\"at
Bielefeld, Postfach 100 131, D-33501 Bielefeld, Germany, and
Departments of Mathematics and Statistics, Purdue University,
150 N. University St., West Lafayette, IN 47907-2067, USA.}
\email{roeckner@math.uni-bielefeld.de}

\date{13 October 2010}

\begin{abstract}
  We establish well-posedness in the mild sense for a class of
  stochastic semilinear evolution equations with a polynomially
  growing quasi-monotone nonlinearity and multiplicative Poisson
  noise. We also study existence and uniqueness of invariant measures
  for the associated semigroup in the Markovian case. A key role is
  played by a new maximal inequality for stochastic convolutions in
  $L_p$ spaces.
\end{abstract}



\thanks{The work for this paper was carried out while the first author
  was visiting the Department of Mathematics of Purdue University
  supported by a grant of the EU. The authors are grateful to an
  anonymous referee for carefully reading the first draft of the
  paper}

\maketitle


\section{Introduction}
The purpose of this paper is to obtain existence and uniqueness of
solutions, as well as existence and uniqueness of invariant measures,
for a class of semilinear stochastic partial differential equations
driven by a discontinuous multiplicative noise. In particular, we
consider the mild formulation of an equation of the type
\begin{equation}     \label{eq:prima}
du(t) + Au(t)\,dt + F(u(t))\,dt = \int_Z G(u(t-),z)\,\m(dt,dz)
\end{equation}
on $L_2(D)$, with $D$ a bounded domain of $\erre^n$. Here $-A$ is the
generator of a strongly continuous semigroup of contractions, $F$ is a
nonlinear function satisfying monotonicity and polynomial growth
conditions, and $\m$ is a compensated Poisson measure. Precise
assumptions on the data of the problem are given in Section
\ref{sec:wp} below. We would like to note that, under appropriate
assumptions on the coefficients, all results of this paper continue to
hold if we add a stochastic term of the type $B(u(t))\,dW(t)$ to the
right hand side of (\ref{eq:prima}), where $W$ is a cylindrical Wiener
process on $L_2(D)$ (see Remark \ref{rmk:ext} below). For
simplicity we concentrate on the jump part of the noise. Similarly,
all results of the paper still hold with minimal modifications if we
allow the functions $F$ and $G$ to depend also on time and to be
random.

While several classes of semilinear stochastic PDEs driven by Wiener
noise, also with rather general nonlinearity $F$, have been
extensively studied (see e.g. \cite{cerrai-libro,DPR-sing,DZ92} and
references therein), a corresponding body of results for equations
driven by jump noise seems to be missing. Let us mention, however,
several notable exceptions: existence of local mild solutions for
equations with locally Lipschitz nonlinearities has been established
in \cite{Kote-Doob} (cf. also \cite{cm:MF10}); stochastic PDEs with
monotone nonlinearities driven by general martingales have been
investigated in \cite{Gyo-semimg} in a variational setting, following
the approach of \cite{KR-spde} (cf. also \cite{cm:IDAQP09} for an ad
hoc method); an analytic approach yielding weak solutions (in the
probabilistic sense) for equations with singular drift and additive
L\'evy noise has been developed in \cite{LR-heat}. The more recent
monograph \cite{PZ-libro} deals also with semilinear SPDEs with
monotone nonlinearity and additive L\'evy noise, and contains a
well-posedness result under a set of regularity assumptions on $F$ and
the stochastic convolution. In particular, continuity with respect to
stronger norms (more precisely, in spaces continuously embedded into
$L_2(D)$) is assumed. We avoid such conditions, thus making our
assumptions more transparent and much easier to verify.

Similarly, not many results are available about the asymptotic
behavior of the solution to SPDEs with jump noise, while the
literature for equations with continuous noise is quite rich (see the
references mentioned above). In this work we show that under a suitably
strong monotonicity assumption one obtains existence, uniqueness, and
ergodicity of invariant measures, while a weaker monotonicity
assumption is enough to obtain the existence of invariant
measures.

Our main contributions could be summarized as follows: we provide a) a
set of sufficient conditions for well-posedness in the mild sense for
SPDEs of the form (\ref{eq:prima}), which to the best of our knowledge
is not contained nor can be derived from existing work; b) a new
concept of generalized mild solution which allows us to treat
equations with a noise coefficient $G$ satisfying only natural
integrability and continuity assumptions; c) existence of invariant
measures without strong dissipativity assumptions on the coefficients
of (\ref{eq:prima}). It is probably worth commenting a little further
on the first issue: it is in general not possible to find a triple $V
\subset H \subset V'$ (see e.g. \cite{Gyo-semimg,KR-spde,PreRoeck} for
details) such that $A+F$ is defined from $V$ to $V'$ and satisfies the
usual continuity, accretivity and coercivity assumptions needed for
the theory to work. For this reason, general semilinear SPDEs cannot
be (always) treated in the variational setting. Moreover, the
Nemitskii operator associated to $F$ is in general not locally
Lipschitz on $L_2(D)$, so one cannot hope to obtain global
well-posedness invoking the local well-posedness results of
\cite{Kote-Doob}, combined with a priori estimates. Finally, while the
analytic approach of \cite{LR-heat} could perhaps be adapted to our
situation, it would cover only the case of additive noise, and
solutions would be obtained only in the sense of the martingale
problem.

The main tool employed in the existence theory is a
Bichteler-Jacod-type inequality for stochastic convolutions on $L_p$
spaces, combined with monotonicity estimates. To obtain well-posedness
for equations with general noise, also of multiplicative type, we need
to relax the concept of solution we work with, in analogy to the
deterministic case (see \cite{BenBre,Bmax}). Finally, we prove
existence of an invariant measure by an argument based on
Krylov-Bogoliubov's theorem under weak dissipativity
conditions. Existence and uniqueness of an invariant measure under
strong dissipativity conditions is also obtained, adapting a classical
method (see e.g. \cite{DZ96}).

The paper is organized as follows. In Section \ref{sec:wp} all
well-posedness results are stated and proved, and Section
\ref{sec:ergo} contains the results on invariant measures. Finally, we
prove in the Appendix an auxiliary result used in Section
\ref{sec:wp}.

Let us conclude this section with a few words about notation. Generic
constants will be denoted by $N$, and we shall use the shorthand
notation $a \lesssim b$ to mean $a \leq Nb$. If the constant $N$
depends on a parameter $p$, we shall also write $N(p)$ and $a
\lesssim_p b$. Given a function $f:\erre\to\erre$, we shall denote its
associated Nemitsky operator by the same symbol. Moreover, given an
integer $k$, we shall write $f^k$ for the function $\xi \mapsto
f(\xi)^k$. For any topological space $X$ we shall denote its Borel
$\sigma$-field by $\mathcal{B}(X)$. We shall occasionally use standard
abbreviations for stochastic integrals with respect to martingales and
stochastic measures, so that $H\cdot X(t) := \int_0^t H(s)\,dX(s)$ and
$\phi \star \mu(t) := \int_0^t\int \phi(s,y)\,\mu(ds,dy)$ (see
e.g. \cite{JacShi} for more details). Given two Banach spaces $E$ and
$F$, we shall denote the set of all functions $f:E \to F$ such that
\[
\sup_{x \neq y} \frac{|f(x)-f(y)|_F}{|x-y|_E} < \infty
\]
by $\dot{C}^{0,1}(E,F)$.

\section{Well-posedness}     \label{sec:wp}
Let $(\Omega,\mathcal{F},(\mathcal{F}_t)_{t \geq 0},\P)$ be a filtered
probability space satisfying the usual conditions and $E$ denote
expectation with respect to $\P$. All stochastic elements will be
defined on this stochastic basis, unless otherwise specified. The
preditable $\sigma$-field will be denoted by $\mathcal{P}$. Let
$(Z,\mathcal{Z},m)$ be a measure space, $\m$ a Poisson measure on
$[0,T] \times Z$ with compensator $\mathrm{Leb} \otimes m$, where Leb
stands for Lebesgue measure. We shall set, for simplicity of notation,
$Z_t = (0,t] \times Z$, for $t \geq 0$, and
$L_p(Z_t):=L_p(Z_t,\mathrm{Leb}\otimes m)$.
Let $D$ be an open bounded subset of $\erre^n$ with smooth boundary
$\partial D$, and set $H=L_2(D)$. The norm and inner product in $H$
are denoted by $|\cdot|$ and $\ip{\cdot}{\cdot}$, respectively, while
the norm in $L_p(D)$, $p \geq 1$, is denoted by $|\cdot|_p$. Given a
Banach space $E$, we shall denote the set of all $E$-valued random
variables $\xi$ such that $\E|\xi|^p<\infty$ by $\mathbb{L}_p(E)$. For
compactness of notation, we also set $\mathbb{L}_p:=\mathbb{L}_p(L_p(D))$.
Moreover, we denote the set of all adapted processes
$u:[0,T]\times\Omega \to H$ such that
\[
|[u]|_p := \Big( \sup_{t \leq T} \E|u(t)|^p \Big)^{1/p} < \infty,
\qquad
\|u\|_p := \Big( \E\sup_{t \leq T} |u(t)|^p \Big)^{1/p} < \infty
\]
by $\mathcal{H}_p(T)$ and $\mathbb{H}_p(T)$, respectively. Note that
$(\mathcal{H}_p(T),|[\,\cdot\,]|_p)$ and
$(\mathbb{H}_p(T),\|\cdot\|_p)$ are Banach spaces. We shall also use
the equivalent norms on $\mathbb{H}_p(T)$ defined by
\[
\|u\|_{p,\alpha} := \Big( 
\E\sup_{t \leq T} e^{-p \alpha t}|u(t)|^p \Big)^{1/p},
\qquad \alpha > 0,
\]
and we shall denote $(\mathbb{H}_p(T),\|\cdot\|_{p,\alpha})$ by
$\mathbb{H}_{p,\alpha}(T)$.

\subsection{Additive noise}
Let us consider the equation
\begin{equation}     \label{eq:raga}
du(t) + Au(t)\,dt + f(u(t))\,dt = \eta u(t)\,dt + \int_Z G(t,z)\,\m(dt,dz),
\qquad u(0)=x,
\end{equation}
where $A$ is a linear maximal monotone operator on $H$;
$f:\erre\to\erre$ is a continuous maximal monotone function satisfying
the growth condition $|f(r)| \lesssim 1+|r|^d$ for some (fixed) $d \in
[1,\infty[$; $G:\Omega \times [0,T] \times Z \times D \to \erre$ is a
$\mathcal{P}\otimes\mathcal{Z}\otimes\mathcal{B}(\erre^n)$-measurable
process, such that $G(t,z) \equiv G(\omega,t,z,\cdot)$ takes values in
$H=L_2(D)$. Finally, $\eta$ is just a constant and the corresponding
term is added for convenience (see below). We shall assume throughout
the paper that the semigroup generated by $-A$ admits a unique
extension to a strongly continuous semigroup of positive contractions
on $L_{2d}(D)$ and $L_{d^*}(D)$, $d^*:=2d^2$. For simplicity of
notation we shall not distinguish among the realizations of $A$ and
$e^{-tA}$ on different $L_p(D)$ spaces, if no confusion can arise.
\begin{rmk}
  Several examples of interest satisfy the assumptions on $A$ just
  mentioned. For instance, $A$ could be chosen as the realization of
  an elliptic operator on $D$ of order $2m$, $m \in \mathbb{N}$, with
  Dirichlet boundary conditions (see e.g. \cite{agmon}). The operator
  $-A$ can also be chosen as the generator of a sub-Markovian strongly
  continuous semigroup of contractions $T_t$ on $L_2(D)$. In fact, an
  argument based on the Riesz-Thorin interpolation theorem shows
  that $T_t$ induces a strongly continuous sub-Markovian contraction
  semigroup $T_t^{(p)}$ on any $L_p(D)$, $p \in [2,+\infty[$ (see
  e.g. \cite[Lemma 1.11]{Ebe} for a detailed proof). The latter class
  of operators includes also nonlocal operators such as, for instance,
  fractional powers of the Laplacian, and even more general
  pseudodifferential operators with negative-definite symbols -- see
  e.g. \cite{Jacob3} for more details and examples.
\end{rmk}

\begin{defi}
  Let $x \in \mathbb{L}_{2d}$. We say that $u \in \mathbb{H}_2(T)$ is a
  mild solution of (\ref{eq:raga}) if $u(t) \in L_{2d}(D)$ $\P$-a.s.
  and
  \begin{equation}     \label{eq:di}
  u(t) = e^{-tA}x 
  + \int_0^t e^{-(t-s)A}\big(\eta u(s) - f(u(s))\big)\,ds
  + \int_{Z_t} e^{-(t-s)A} G(s,z)\,\m(ds,dz)
  \end{equation}
  $\P$-a.s. for all $t \in [0,T]$, and all integrals on the right-hand
  side exist. 
\end{defi}
Let us denote the class of processes $G$ as above such that
\[
\E \int_0^T\!\!\Big[\int_Z |G(t,z)|_{p}^p\,m(dz)
+ \Big(\int_Z |G(t,z)|_{p}^2\,m(dz)\Big)^{p/2}\Big]\,dt
< \infty.
\]
by $\mathcal{L}_p$. Setting $d^*=2d^2$, we shall see below that a
sufficient condition for the existence of the integrals appearing in
(\ref{eq:di}) is that $G \in \mathcal{L}_{d^*}$. This also explains
the condition imposed on the sequence $\{G_n\}$ in the next
definition.
\begin{defi}
  Let $x \in \mathbb{L}_2$. We say that $u \in \mathbb{H}_2(T)$ is a
  generalized mild solution of (\ref{eq:raga}) if there exist a
  sequence $\{x_n\} \subset \mathbb{L}_{2d}$ and a sequence
  $\{G_n\} \subset \mathcal{L}_{d^*}$ with $x_n \to x$ in
  $\mathbb{L}_2$ and $G_n \to G$ in $\mathbb{L}_2(L_2(Z_T))$, such
  that $u_n \to u$ in $\mathbb{H}_2(T)$, where $u_n$ is the mild
  solution of (\ref{eq:raga}) with $x_n$ and $G_n$ replacing $x$ and
  $G$, respectively.
\end{defi}
In order to establish well-posedness of the stochastic equation, we
need the following maximal inequalities, that are extensions to a
(specific) Banach space setting of the corresponding inequalities
proved for Hilbert space valued processes in \cite{cm:JFA10}, with a
completely different proof.
\begin{lemma}     \label{lm:31}
  Let $E=L_p(D)$, $p \in [2,\infty)$. Assume that $g:\Omega \times
  [0,T] \times Z \times D \to \erre$ is a $\mathcal{P} \otimes
  \mathcal{Z} \otimes \mathcal{B}(\erre^n)$-measurable function such
  that the expectation on the right-hand side of (\ref{eq:BJ}) below
  is finite. Then there exists a constant $N=N(p,T)$ such that
  \begin{multline}     \label{eq:BJ} 
    \E \sup_{t\leq T} \Big|
    \int_0^t\!\!\int_Z g(s,z)\,\bar\mu(ds,dz)\Big|_E^p \\
    \leq N \E \int_0^T
    \Big[ \int_Z |g(s,z)|_E^p\,m(dz) + \Big(\int_Z
    |g(s,z)|_E^2\,m(dz)\Big)^{p/2} \Big]\,ds,
  \end{multline}
  where $(p,T)\mapsto N$ is continuous.
  Furthermore, let $-A$ be the generator of a strongly continuous
  semigroup $e^{-tA}$ of positive contractions on $E$. Then one also
  has
  \begin{multline}
    \label{eq:BJconv}
    \E \sup_{t\leq T} \Big| 
    \int_0^t\int_Z e^{-(t-s)A} g(s,z)\,\bar\mu(ds,dz)\Big|_E^p\\
    \leq
    N \E \int_0^T \Big[ \int_Z |g(s,z)|_E^p\,m(dz)
    + \Big(\int_Z |g(s,z)|_E^2\,m(dz)\Big)^{p/2} \Big]\,ds,
  \end{multline}
  where $N$ is the same constant as in (\ref{eq:BJ}).
\end{lemma}
\begin{proof}
  We proceed in several steps.

  \noindent \textsl{Step 1.} Let us assume that $m(Z)<\infty$ (this
  hypothesis will be removed in the next step). Note that, by Jensen's
  (or H\"older's) inequality and Fubini's theorem, one has
  \begin{align*}
    \int_D\!\E \int_0^T\!\Big( \int_Z |g(s,z,\xi)|^2\,m(dz)
    \Big)^{p/2}ds\,d\xi &\lesssim \int_D\!\E \int_0^T\!\!\int_Z
    |g(s,z,\xi)|^p\,m(dz)\,ds\,d\xi\\
    &=\E \int_0^T\!\!\int_Z |g(s,z)|^p_E\,m(dz)\,ds < \infty,
  \end{align*}
  therefore, since the right-hand side of (\ref{eq:BJ}) is finite,
  Fubini's theorem implies that 
  \[
  \E \int_0^T \Big[ \int_Z |g(s,z,\xi)|^p\,m(dz) + \Big(\int_Z
  |g(s,z,\xi)|^2\,m(dz)\Big)^{p/2} \Big]\,ds < \infty
  \]
  for a.a. $\xi \in D$. By the Bichteler-Jacod
  inequality for real-valued integrands (see e.g. \cite{BGJ,cm:JFA10})
  we have
  \begin{multline}
    \E \Big| \int_0^T\!\!\!\int_Z g(s,z,\xi)\,\m(ds,dz)\Big|^p \\
    \lesssim_{p,T} \E \int_0^T \Big[ \int_Z |g(s,z,\xi)|^p\,m(dz) +
    \Big(\int_Z |g(s,z,\xi)|^2\,m(dz)\Big)^{p/2} \Big]\,ds
  \end{multline}
  for a.a. $\xi \in D$. Furthermore, Fubini's theorem for integrals
  with respect to random measures (see e.g. \cite{lebedev} or
  \cite[App.~A]{BDMKR}) yields
  \[
  \E \Big| \int_0^T\!\!\!\int_Z g(s,z)\,\m(ds,dz)\Big|^p_{E} = \int_D
  \E \Big| \int_0^T\!\!\!\int_Z g(s,z,\xi)\,\m(ds,dz)\Big|^p\,d\xi,
  \]
  hence also
  \begin{align*}
    \E \Big| \int_0^T\!\!\!\int_Z g(s,z)\,\m(ds,dz)\Big|^p_{E}
    \lesssim_{p,T}\;
    &\E \int_0^T\!\!\!\int_Z\!\int_D |g(s,z,\xi)|^p\,d\xi\,m(dz)\,ds\\
    & + \E\int_0^T\!\!\!\int_D \Big(\int_Z |g(s,z,\xi)|^2\,m(dz)
    \Big)^{p/2}\,d\xi\,ds.
  \end{align*}
  Minkowski's inequality (see e.g. \cite[Thm.~2.4]{LiLo}) implies that
  the second term on the right-hand side of the previous inequality is
  less than or equal to
  \[
  \E\int_0^T\!\!\Big( \int_Z \Big( \int_D
  |g(s,z,\xi)|^p\,d\xi\Big)^{2/p} m(dz) \Big)^{p/2}ds =
  \E\int_0^T\!\!\Big(\int_Z |g(s,z)|^2_{E}\,m(dz)\Big)^{p/2}ds.
  \]
  We have thus proved that
  \[
  \E |g \star \m(T)|_E^p \lesssim_{p,T} \E \int_0^T \Big[ \int_Z
  |g(s,z)|_E^p\,m(dz) + \Big(\int_Z |g(s,z)|_E^2\,m(dz)\Big)^{p/2}
  \Big]\,ds.
  \]
  \smallskip\par\noindent
  \textsl{Step 2.} Let us turn to the general case $m(Z)=\infty$. Let
  $\{Z_n\}_{n\in\mathbb{N}}$ a sequence of subsets of $Z$ such that
  $\cup_{n\in\mathbb{N}} Z_n = Z$, $Z_n \subset Z_{n+1}$ and
  $m(Z_n)<\infty$ for all $n \in \mathbb{N}$. By the Bichteler-Jacod
  inequality for real-valued integrands we have
  \begin{align*}
    &\E\Big| \int_0^T\!\!\int_Z g(s,z,\xi)\mathbf{1}_{Z_n}(z)\,\m(ds,dz)
      \Big|^p\\
    &\quad \lesssim_{p,T} \E\int_0^T\!\!\int_Z
     |g(s,z,\xi)|^p \mathbf{1}_{Z_n}(z)\,m(dz)\,ds
     + \E\int_0^T\!\Big(\int_Z |g(s,z,\xi)|^2 \mathbf{1}_{Z_n}(z)\,m(dz)
         \Big)^{p/2}\,ds\\
    &\quad = \E\int_0^T\!\!\int_Z
     |g(s,z,\xi)|^p\,m_n(dz)\,ds
     + \E\int_0^T\!\Big(\int_Z |g(s,z,\xi)|^2\,m_n(dz)
       \Big)^{p/2}\,ds,
  \end{align*}
  where $m_n(\cdot):=m(\cdot \cap Z_n)$. Integrating both sides of
  this inequality with respect to $\xi$ over $D$ we obtain, using
  Fubini's theorem and Minkowski's inequality, we are left with
  \begin{align*}
    &\E \left| \int_0^T\!\!\int_Z g_n(s,z)\,\m(ds,dz)
       \right|_E^p\\
    &\qquad \lesssim_{p,T}
       \E\int_0^T\!\!\int_Z \big| g_n(s,z) \big|_E^p\,m(dz)\,ds
     + \E\int_0^T\!\Big(\int_Z \big| g_n(s,z)
            \big|_E^2\,m(dz)\Big)^{p/2}\,ds\\
    &\qquad \leq   \E\int_0^T\!\!\int_Z \big| g(s,z) \big|_E^p\,m(dz)\,ds
     + \E\int_0^T\!\Big( \int_Z
       \big|g(s,z)\big|_E^2\,m(dz)\Big)^{p/2}\,ds,
  \end{align*}
  where $g_n(\cdot,z):=g(\cdot,z)\mathbf{1}_{Z_n}(z)$.

  Let us now prove that $g_n \star \m(T)$ converges to $g \star \m(T)$
  on $D \times \Omega$ in $\mathrm{Leb} \otimes \P$-measure as $n \to
  \infty$. In fact, by the isometric formula for stochastic integrals
  with respect to compensated Poisson measures, we have
  \[
  \big| (g_n-g) \star \m(T) \big|^2_{L_2(D \times \Omega)} = \E\int_{D
    \times [0,T] \times Z} \big| g_n(s,z,\xi) - g(s,z,\xi)
  \big|^2\,m(dz)\,ds\,d\xi,
  \]
  which converges to zero as $n \to \infty$ by the dominated
  convergence theorem. In fact, $g_n \uparrow g$ a.e. on $D \times
  [0,T] \times Z$, $\P$-a.s., and
  \begin{align*}
  &\E\int_{D \times [0,T] \times Z} \big| g_n(s,z,\xi) - g(s,z,\xi)
  \big|^2\,m(dz)\,ds\,d\xi\\
  &\qquad \leq 2\E \int_0^T \!\! \int_Z \big| g(s,z)
  \big|^2_{L_2(D)}\,m(dz)\,ds
  \lesssim \E \int_0^T \!\! \int_Z \big| g(s,z)
  \big|^2_E\,m(dz)\,ds < \infty.
  \end{align*}
  Finally, by Fatou's lemma, we have
  \begin{align*}
    \E\big| g \star \m(T) \big|_E^p &=
    \E \int_D \big| (g \star \m(T))(\xi) \big|^p\,d\xi\\
    &\leq \liminf_{n \to \infty} \E \int_D \big| (g_n \star
    \m(T))(\xi) \big|^p\,d\xi
    = \liminf_{n \to \infty} \E\big| g_n \star \m(T) \big|_E^p\\
    &\leq \E\int_0^T\!\!\int_Z \big| g(s,z) \big|_E^p\,m(dz)\,ds +
    \E\int_0^T\!\Big( \int_Z
    \big|g(s,z)\big|_E^2\,m(dz)\Big)^{p/2}\,ds.
  \end{align*}
  \textsl{Step 3.}  Estimate (\ref{eq:BJ}) now follows immediately, by
  Doob's inequality, provided we can prove that $g \star \m$ is an
  $E$-valued martingale. For this it suffices to prove that
  \[
  \E\big[\ip{g \star \m(t)-g \star \m(s)}{\phi}\big|\mathcal{F}_s\big]
  = 0, \qquad 0 \leq s \leq t \leq T,
  \]
  for all $\phi \in C^\infty_c(D)$, the space of infinitely
  differentiable functions with compact support on $D$. In fact, we
  have, by the stochastic Fubini theorem,
  \begin{align*}
    \ip{g \star \m(t)-g \star \m(s)}{\phi} &= \Big\langle
    \int_{(s,t]}\!\int_Z g(r,z)\,\m(dr,dz),\phi\Big\rangle\\
    &= \int_{(s,t]}\!\int_Z\!\int_D
    g(r,z,\xi)\phi(\xi)\,d\xi\,\m(dr,dz),
  \end{align*}
  where the last term has $\mathcal{F}_s$-conditional expectation
  equal to zero by well-known properties of Poisson measures. In order
  for the above computation to be rigorous, we need to show that the
  last stochastic integral is well defined: using H\"older's
  inequality and recalling that $g \in \mathcal{L}_p$, we get
  \begin{align*}
    &\E\int_{(s,t]}\!\int_Z \Big[\int_D
    g(r,z,\xi)\phi(\xi)\,d\xi\Big]^2\,m(dz)\,dr
    \leq |\phi|^2_{\frac{p}{p-1}} \E\int_0^T \int_Z |g(s,z)|_E^2\,m(dz)\,ds\\
    &\qquad \leq |\phi|^2_{\frac{p}{p-1}}T^{p/(p-2)} \Big( \E\int_0^T
    \Big(\int_Z |g(s,z)|_p^2\,m(dz)\Big)^{p/2}\,ds \Big)^{2/p} <
    \infty.
  \end{align*}
  \smallskip\par\noindent
  \textsl{Step 4.}  In order to extend the result to stochastic
  convolutions, we need a dilation theorem due to Fendler
  \cite[Thm.~1]{Fendler}. In particular, there exist a measure space
  $(Y,\mathcal{A},n)$, a strongly continuous group of isometries
  $T(t)$ on $\bar{E}:=L_p(Y,n)$, an isometric linear embedding $j:
  L_p(D) \to L_p(Y,n)$, and a contractive projection $\pi:L_p(Y,n) \to
  L_p(D)$ such that $j \circ e^{tA} = \pi \circ T(t) \circ j$ for all
  $t \geq 0$. Then we have, recalling that the operator norms of $\pi$
  and $T(t)$ are less than or equal to one,
  \begin{eqnarray*}
    \lefteqn{\E \sup_{t\leq T} 
      \Big| \int_0^t\int_Z e^{-(t-s)A} g(s,z)\,\bar\mu(ds,dz)\Big|^p_E}\\
    &=& 
    \E \sup_{t\leq T} \Big|
    \pi T(t) \int_0^t\int_Z T(-s) j(g(s,z))\,\bar\mu(ds,dz)\Big|^p_{\bar{E}}\\
    &\leq& |\pi|^p \; \sup_{t\leq T} |T(t)|^p \;
    \E \sup_{t\leq T} \Big|
    \int_0^t\int_Z T(-s) j(g(s,z))\,\bar\mu(ds,dz)\Big|^p_{\bar{E}}\\
    &\leq& \E \sup_{t\leq T} \Big|
    \int_0^t\int_Z T(-s) j(g(s,z))\,\bar\mu(ds,dz)\Big|^p_{\bar{E}}
  \end{eqnarray*}
  Now inequality (\ref{eq:BJ}) implies that there exists a constant
  $N=N(p,T)$ such that
  \begin{eqnarray*}
    \lefteqn{\E \sup_{t\leq T} 
      \Big| \int_0^t\int_Z e^{-(t-s)A} g(s,z)\,\bar{\mu}(ds,dz)\Big|_E^p}\\
    &\leq& N \E \int_0^T \Big[ \int_Z |T(-s)j(g(s,z))|_{\bar{E}}^p \,m(dz)
    + \Big(\int_Z |T(-s)j(g(s,z))|_{\bar{E}}^2 \,m(dz)\Big)^{p/2} \Big]\,ds\\
    &\leq& N \E \int_0^T \Big[ \int_Z |g(s,z)|_E^p\,m(dz)
    + \Big(\int_Z |g(s,z)|_E^2\,m(dz)\Big)^{p/2} \Big]\,ds
  \end{eqnarray*}
  where we have used again that $T(t)$ is a unitary group and that the
  norms of $\bar{E}$ and $E$ are equal.
\end{proof}
\begin{rmk}

  (i) The idea of using dilation theorems to extend results
  from stochastic integrals to stochastic convolutions has been
  introduced, to the best of our knowledge, in \cite{HauSei}.

  \noindent (ii) Since $g \star \bar{\mu}$ is a martingale taking
  values in $L_{p}(D)$, it has a c\`adl\`ag modification, as it
  follows by a theorem of Brooks and Dinculeanu (see
  \cite[Thm.~3]{BroDinc87}).  Moreover, the stochastic convolution
  also admits a c\`adl\`ag modification by the dilation method, as in
  \cite{HauSei} or \cite[p.~161]{PZ-libro}.
\end{rmk}

We shall need to regularize the monotone nonlinearity $f$ by its
Yosida approximation $f_\lambda$, $\lambda > 0$. In particular, let
$J_\lambda(x)=(I+\lambda f)^{-1}(x)$,
$f_\lambda(x)=\lambda^{-1}(x-J_\lambda(x))$. It is well known that
$f_\lambda(x)=f(J_\lambda(x))$ and $f_\lambda \in \dot{C}^{0,1}(\erre)$ with
Lipschitz constant bounded by $2/\lambda$. For more details on maximal
monotone operators and their approximations see e.g. \cite{barbu,
Bmax}. Let us consider the regularized equation
\begin{equation}     \label{eq:ragayo}
  du(t) + Au(t)\,dt + f_\lambda(u(t))\,dt = \eta u(t)\,dt
  + \int_Z G(t,z)\,\m(dt,dz),
  \qquad u(0)=x,
\end{equation}
which admits a unique c\`adl\`ag mild solution $u_\lambda \in
\mathbb{H}_2(T)$ because $-A$ is the generator of a strongly
continuous semigroup of contractions and $f_\lambda$ is Lipschitz (see
e.g. \cite{Kote-Doob,cm:JFA10,PZ-libro}).

We shall now establish an a priori estimate for solutions of the
regularized equations.
\begin{lemma}     \label{lm:apri}
  Assume that $x \in \mathbb{L}_{2d}$ and $G \in \mathcal{L}_{d^*}$.
  Then there exists a constant $N=N(T,d,\eta,|D|)$ such that
  \begin{equation}     \label{eq:apri}
  \E \sup_{t \leq T} |u_\lambda(t)|_{2d}^{2d} \leq 
  N \big( 1 + \E|x|_{2d}^{2d} \big).
  \end{equation}
\end{lemma}
\begin{proof}
  We proceed by the technique of ``subtracting the stochastic
  convolution'': set
  \[
  y_\lambda(t) = u_\lambda(t) - \int_0^t e^{-(t-s)A}G(s,z)\,\m(ds,dz)
  =: u_\lambda(t) - G_A(t), \qquad t \leq T,
  \]
  where
  \[
  G_A(t) := \int_0^t\!\!\int_Z e^{-(t-s)A}G(s,z)\,\m(ds,dz).
  \]
  Then $y_\lambda$ is also a mild solution in $L_2(D)$ of the
  deterministic equation with random coefficients
  \begin{equation}     \label{eq:alter}
  y'_\lambda(t) + Ay_\lambda(t) + f_\lambda(y_\lambda(t)+G_A(t))
  = \eta y_\lambda(t) + \eta G_A(t),
  \qquad y_\lambda(0)=x,
  \end{equation}
  $\P$-a.s., where $\phi'(t):=d\phi(t)/dt$. We are now going to prove
  that $y_\lambda$ is also a mild solution of (\ref{eq:alter}) in
  $L_{2d}(D)$. Setting
  \[
  \tilde{f}_\lambda(t,y) := f_\lambda(y+G_A(t)) - \eta(y+G_A(t))
  \]
  and rewriting (\ref{eq:alter}) as
  \[
  y'_\lambda(t) + Ay_\lambda(t) + \tilde{f}_\lambda(t,y_\lambda(t)) =
  0,
  \]
  we conclude that (\ref{eq:alter}) admits a unique mild solution in
  $L_{2d}(D)$ by Proposition \ref{prop:pazzo} below (see the Appendix).

  Let $y_{\lambda\beta}$ be the strong solution in $L_{2d}(D)$ of the
  equation
  \begin{equation}     \label{eq:ego}
  y'_{\lambda\beta}(t) + A_\beta y_{\lambda\beta}(t) 
  + f_\lambda(y_{\lambda\beta}(t)+G_A(t))
  = \eta y_{\lambda\beta}(t) + \eta G_A(t),
  \qquad y_\lambda(0)=x,
  \end{equation}
  which exists and is unique because the Yosida approximation
  $A_\beta$ is a bounded operator on $L_{2d}(D)$.
  Let us recall that the duality map $J:L_{2d}(D) \to
  L_{\frac{2d}{2d-1}}(D)$ is single valued and defined by
  \[
  J(\phi): \xi \mapsto |\phi(\xi)|^{2d-2} \phi(\xi) |\phi|_{2d}^{2-2d}
  \]
  for almost all $\xi \in D$. Moreover, since $L_{\frac{2d}{2d-1}}(D)$
  is uniformly convex, $J(\phi)$ coincides with the G\^ateaux
  derivative of $\phi \mapsto |\phi|_{2d}^2/2$. Therefore, multiplying
  (in the sense of the duality product of $L_{2d}(D)$ and
  $L_{\frac{2d}{2d-1}}(D)$) both sides of (\ref{eq:ego}) by the
  function
  \[
  J(y_{\lambda\beta}(t)) |y_{\lambda\beta}(t)|_{2d}^{2d-2} =
  |y_{\lambda\beta}(t)|^{2d-2} y_{\lambda\beta}(t),
  \]
  we get
  \begin{align*}
    &\frac{1}{2d} \frac{d}{dt} |y_{\lambda\beta}(t)|_{2d}^{2d}
     + \ip{A_\beta y_{\lambda\beta}(t)}{J(y_{\lambda\beta}(t))}
                              |y_{\lambda\beta}(t)|_{2d}^{2d-2}\\
    &\qquad
    + \ip{f_\lambda(y_{\lambda\beta}(t)+G_A(t))}
         {|y_{\lambda\beta}(t)|^{2d-2} y_{\lambda\beta}(t)}\\
    &\qquad\quad
    = \eta |y_{\lambda\beta}(t)|^{2d}_{2d} +
    \eta\ip{|y_{\lambda\beta}(t)|^{2d-2} y_{\lambda\beta}(t)}{G_A(t)}.
  \end{align*}
  Since $A$ is $m$-accretive in $L_{2d}(D)$ (more precisely, $A$ is an
  $m$-accretive subset of $L_{2d}(D) \times L_{2d}(D)$), its Yosida
  approximation $A_\beta=A(I+\beta A)^{-1}$ is also $m$-accretive (see
  e.g. \cite[Prop.~2.3.2]{barbu}), thus the second term on the left
  hand side is positive because $J$ is single-valued. Moreover, we
  have, omitting the dependence on $t$ for simplicity of notation,
  \begin{align*}
    f_{\lambda}(y_{\lambda\beta}+G_A)|y_{\lambda\beta}|^{2d-2}y_{\lambda\beta}
      &= \big( f_\lambda(y_{\lambda\beta}+G_A)
         - f_\lambda(G_A)\big) y_{\lambda\beta}
    |y_{\lambda\beta}|^{2d-2}\\
    &\quad + f_\lambda(G_A) |y_{\lambda\beta}|^{2d-2}y_{\lambda\beta}\\
    &\geq f_\lambda(G_A) |y_{\lambda\beta}|^{2d-2}y_{\lambda\beta}
    \qquad (t,\xi)\text{-a.e.},
  \end{align*}
  as it follows by the monotonicity of $f_\lambda$. Therefore we can
  write
  \begin{align*}
    \frac{1}{2d} \; \frac{d}{dt} |y_{\lambda\beta}(t)|_{2d}^{2d} &\leq \eta
    |y_{\lambda\beta}(t)|_{2d}^{2d}
    + \ip{\eta G_A(t) - f_\lambda(G_A(t))}
         {|y_{\lambda\beta}(t)|^{2d-2} y_{\lambda\beta}(t)}\\
    &\leq \eta |y_{\lambda\beta}(t)|_{2d}^{2d}
    + |\eta G_A(t) - f_\lambda(G_A)|_{2d} \;
    \big| |y_{\lambda\beta}(t)|^{2d-1} \big|_{\frac{2d}{2d-1}}\\
    &= \eta |y_{\lambda\beta}(t)|_{2d}^{2d}
    + |\eta G_A(t) - f_\lambda(G_A)|_{2d} \; |y_{\lambda\beta}(t)|_{2d}^{2d-1}\\
    &\leq \eta |y_{\lambda\beta}(t)|_{2d}^{2d} + \frac{1}{2d} |\eta G_A(t) -
    f_\lambda(G_A)|_{2d}^{2d} + \frac{2d-1}{2d}
    |y_{\lambda\beta}(t)|_{2d}^{2d},
  \end{align*}
  where we have used H\"older's and Young's inequalities with
  conjugate exponents $2d$ and $2d/(2d-1)$. A simple computation
  reveals immediately that there exists a constant $N$ depending only
  on $d$ and $\eta$ such that
  \[
  |\eta G_A(t) - f_\lambda(G_A)|_{2d}^{2d} \leq N (1 +
  |G_A(t)|_{2d^2}^{2d^2}).
  \]
  We thus arrive at the inequality
  \[
  \frac{1}{2d} \frac{d}{dt} |y_{\lambda\beta}(t)|_{2d}^{2d} \leq
  \big( \eta + \frac{2d-1}{2d} \big) |y_{\lambda\beta}(t)|_{2d}^{2d}
  + N\big( 1 + |G_A(t)|_{2d^2}^{2d^2} \big),
  \]
  and Gronwall's inequality yields
  \[
  |y_{\lambda\beta}(t)|_{2d}^{2d} \lesssim_{d,\eta} 1 + |x|_{2d}^{2d}
  + |G_A(t)|_{2d^2}^{2d^2},
  \]
  hence also, thanks to (\ref{eq:BJconv}) and the hypothesis that $G
  \in \mathcal{L}_{d^*}$,
  \[
  \E \sup_{t \leq T} |y_{\lambda\beta}(t)|_{2d}^{2d} 
  \leq N(1 + \E |x|_{2d}^{2d}).
  \]
  where the constant $N$ does not depend on $\lambda$. Let us now
  prove that $y_{\lambda\beta} \to y_\lambda$ in $\mathbb{H}_2(T)$ as
  $\beta \to 0$: we have
  \begin{align*}
  |y_{\lambda\beta}(t) - y_\lambda(t)| &\leq
  |(e^{-tA_\beta}-e^{-tA})y_\lambda(0)|\\
  &\quad + \int_0^t \big|e^{-(t-s)A_\beta} \tilde{f}_\lambda(s,y_{\lambda\beta}(s))
    - e^{-(t-s)A} \tilde{f}_\lambda(s,y_\lambda(s))\big|\,ds\\
  &\leq
  |(e^{-tA_\beta}-e^{-tA})y_\lambda(0)|\\
  &\quad + \int_0^t \big|\big( e^{-(t-s)A_\beta} - e^{-(t-s)A} \big)\,
   \tilde{f}_\lambda(s,y_{\lambda}(s)) \big|\,ds\\
  &\quad + \int_0^t |e^{-(t-s)A_\beta}|\,|\tilde{f}_\lambda(s,y_{\lambda\beta}(s))
    - \tilde{f}_\lambda(s,y_\lambda(s))|\,ds\\
  &=: I_{1\beta}(t) + I_{2\beta}(t) + I_{3\beta}(t).
  \end{align*}
  By well-known properties of the Yosida approximation we have
  \[
  \sup_{t \leq T} I_{1\beta}(t)^2 =
  \sup_{t \leq T} \big| e^{-t A_\beta} y_\lambda(0) 
  - e^{-tA} y_\lambda(0) \big|^2 \to 0
  \]
  $\P$-a.s. as $\beta \to 0$, and
  \[
  \sup_{t \leq T} I_{1\beta}(t)^2 \lesssim |y_\lambda(0)|^2 \in \mathbb{L}_2,
  \]
  therefore, by the dominated convergence theorem, the expectation of
  the left-hand side of the previous expression converges to zero as
  $\beta \to 0$.  Similarly, we also have
  \[
  \sup_{t \leq T} I_{2\beta}(t)^2
  \lesssim \int_0^T \sup_{t \leq T} \big|e^{-t A_\beta}
  \tilde{f}_\lambda(s,y_{\lambda}(s))
  - e^{-tA} \tilde{f}_\lambda(s,y_\lambda(s))\big|^2\,ds,
  \]
  where the integrand on the right-hand side converges to zero
  $\P$-a.s. for all $s \leq T$. The last inequality also yields,
  recalling that $y_\lambda$ and $G_A$ belong to $\mathbb{H}_2(T)$,
  \[
  \sup_{t \leq T} I_{2\beta}(t)^2 \lesssim 
  \int_0^T | \tilde{f}_\lambda(s,y_\lambda(s)) |^2\,ds \in \mathbb{L}_2,
  \]
  hence, by the dominated convergence theorem, the expectation of the
  left-hand side of the previous expression converges to zero as
  $\beta \to 0$. Finally, since $A_\beta$ generates a contraction
  semigroup, by definition of $\tilde{f}_\lambda$ and the fact that
  $f_\lambda$ has Lipschitz constant bounded by $2/\lambda$, we have
  \[
  \E \sup_{t \leq T} I_{3\beta}(t)^2 \leq (2/\lambda + \eta)
  \int_0^T \E \sup_{s\leq t}
      |y_{\lambda\beta}(s) - y_\lambda(s)|^2\,dt.
  \]
  Writing
  \[
  \E \sup_{t\leq T} |y_{\lambda\beta}(t) - y_\lambda(t)|^2 \lesssim
  \E \sup_{t \leq T} I_{1\beta}(t)^2 +
  \E \sup_{t \leq T} I_{2\beta}(t)^2 +
  \E \sup_{t \leq T} I_{3\beta}(t)^2,
  \]
  using the above expressions, Gronwall's lemma, and letting $\beta
  \to 0$, we obtain the claim. Therefore, by a lower semicontinuity
  argument, we get
  \[
  \E \sup_{t \leq T} |y_{\lambda}(t)|_{2d}^{2d} 
  \leq N(1 + \E |x|_{2d}^{2d}).
  \]
  By definition of $y_\lambda$ we also infer that
  \[
  \E \sup_{t \leq T} |u_\lambda(t)|_{2d}^{2d} \lesssim_d
  \E \sup_{t \leq T} |y_\lambda(t)|_{2d}^{2d}
  + \E \sup_{t \leq T} |G_A(t)|_{2d}^{2d}.
  \]
  Since
  \[
  \E \sup_{t \leq T} |G_A(t)|_{2d}^{2d} \lesssim_{|D|}
  \E \sup_{t \leq T} |G_A(t)|_{2d^2}^{2d} 
  \lesssim
  1 + \E \sup_{t \leq T} |G_A(t)|_{2d^2}^{2d^2},
  \]
  we conclude
  \[
  \E \sup_{t \leq T} |u_\lambda(t)|_{2d}^{2d} \lesssim_{T,d,\eta,|D|} 
  1 + \E|x|_{2d}^{2d}. \qedhere
  \]
\end{proof}
The a priori estimate just obtained for the solution of the
regularized equation allows us to construct a mild solution of the
original equation as a limit in $\mathbb{H}_2(T)$, as the following
proposition shows.
\begin{prop}     \label{prop:tano}
  Assume that $x \in \mathbb{L}_{2d}$ and $G \in \mathcal{L}_{d^*}$.
  Then equation (\ref{eq:raga}) admits a unique c\`adl\`ag mild
  solution in $\mathbb{H}_2(T)$ which satisfies the estimate
  \[
  \E \sup_{t \leq T} |u(t)|_{2d}^{2d} \leq N(1 + \E|x|_{2d}^{2d}) 
  \]
  with $N=N(T,d,\eta,|D|)$. Moreover, we have $x \mapsto u(x) \in
  \dot{C}^{0,1}(\mathbb{L}_2,\mathbb{H}_2(T))$.
\end{prop}
\begin{proof}
  Let $u_\lambda$ be the solution of the regularized equation
  (\ref{eq:ragayo}), and $u_{\lambda\beta}$ be the strong solution of
  (\ref{eq:ragayo}) with $A$ replaced by $A_\beta$ studied in the
  proof of Lemma \ref{lm:apri} (or see \cite[Thm.~34.7]{Met}). Then
  $u_{\lambda\beta}-u_{\mu\beta}$ solves $\P$-a.s. the equation
  \begin{multline}     \label{eq:postera}
  \frac{d}{dt}(u_{\lambda\beta}(t)-u_{\mu\beta}(t))
  + A_\beta(u_{\lambda\beta}(t)-u_{\mu\beta}(t)) \\
  + f_\lambda(u_{\lambda\beta}(t)) - f_\mu(u_{\mu\beta}(t))
  = \eta (u_{\lambda\beta}(t)-u_{\mu\beta}(t)).
  \end{multline}
  Note that we have
  \begin{align*}
  u_{\lambda\beta} - u_{\mu\beta} &= u_{\lambda\beta} - J_\lambda u_{\lambda\beta}
  + J_\lambda u_{\lambda\beta} - J_\mu u_{\mu\beta}
  + J_\mu u_{\mu\beta} - u_{\mu\beta}\\
  &= \lambda f_\lambda(u_{\lambda\beta}) + J_\lambda u_{\lambda\beta}
     - J_\mu u_{\mu\beta} - \mu f_\mu(u_{\mu\beta}),
  \end{align*}
  hence, recalling that $f_\lambda(u_{\lambda\beta})=f(J_\lambda u_{\lambda\beta})$,
  \begin{align*}
  \ip{f_\lambda(u_{\lambda\beta})-f_\mu(u_{\mu\beta})}{u_{\lambda\beta}-u_{\mu\beta}}
  &\geq \ip{f_\lambda(u_{\lambda\beta}) - f_\mu(u_{\mu\beta})}%
  {\lambda f_\lambda(u_{\lambda\beta}) - \mu f_\mu(u_{\mu\beta})}\\
  &\geq \lambda |f_\lambda(u_{\lambda\beta})|^2 
  + \mu |f_\mu(u_{\mu\beta})|^2 - (\lambda + \mu) |f_\lambda(u_{\lambda\beta})|
        |f_\mu(u_{\mu\beta})|\\
  &\geq -\frac{\mu}{2} |f_\lambda(u_{\lambda\beta})|^2
   -\frac{\lambda}{2}|f_\mu(u_{\mu\beta})|^2,
  \end{align*}
  thus also, by the monotonicity of $A$,
  \[
  \frac{d}{dt} |u_{\lambda\beta}(t)-u_{\mu\beta}(t)|^2 - 2 \eta
  |u_{\lambda\beta}(t)-u_{\mu\beta}(t)|^2 \leq
  \mu |f_\lambda(u_{\lambda\beta}(t))|^2 + \lambda |f_\mu(u_{\mu\beta}(t))|^2.
  \]
  Multiplying both sides by $e^{-2 \eta t}$ and integrating we get
  \[
  e^{-2\eta t} |u_{\lambda\beta}(t)-u_{\mu\beta}(t)|^2 \leq \int_0^t e^{-2\eta s}
  \big( \mu |f_\lambda(u_{\lambda\beta}(s))|^2 + \lambda |f_\mu(u_{\mu\beta}(s))|^2
  \big)\,ds.
  \]
  Since $u_{\lambda\beta} \to u_\lambda$ in $\mathbb{H}_2(T)$ as
  $\beta \to 0$ (as shown in the proof of Lemma \ref{lm:apri}) and
  $f_\lambda$ is Lipschitz, we can pass to the limit as $\beta \to 0$
  in the previous equation, which then holds with $u_{\lambda\beta}$
  and $u_{\mu\beta}$ replaced by $u_\lambda$ and $u_\mu$,
  respectively. Taking supremum and expectation we thus arrive at
  \[
  \E \sup_{t\leq T} |u_\lambda(t)-u_\mu(t)|^2 \leq e^{2\eta T} T
  (\lambda+\mu) \E\sup_{t\leq T} \big(|f_\lambda(u_\lambda(t))|^2 +
  |f_\mu(u_\mu(t))|^2\big).
  \]
  Recalling that $|f_\lambda(x)|\leq |f(x)|$ for all $x\in\erre$,
  Lemma \ref{lm:apri} yields
\begin{equation}     \label{eq:ori}
\E\sup_{t\leq T} |f_\lambda(u_\lambda(t))|^2 \leq
\E\sup_{t\leq T} |f(u_\lambda(t))|^2 \lesssim 
\E\sup_{t\leq T} |u_\lambda(t)|^{2d}_{2d} \leq N(1 + \E|x|_{2d}^{2d}),
\end{equation}
where the constant $N$ does not depend on $\lambda$, hence
\[
\E \sup_{t \leq T} |u_\lambda(t)-u_\mu(t)|^2 \lesssim_T
(\lambda + \mu) \big( 1 + \E|x|_{2d}^{2d} \big),
\]
which shows that $\{u_\lambda\}$ is a Cauchy sequence in
$\mathbb{H}_2(T)$, and in particular there exists $u \in
\mathbb{H}_2(T)$ such that $u_\lambda \to u$ in
$\mathbb{H}_2(T)$. Moreover, since $u_\lambda$ is c\`adl\`ag and the
subset of c\`adl\`ag processes in $\mathbb{H}_2(T)$ is closed, we
infer that $u$ is itself c\`adl\`ag.

Recalling that $f_\lambda(x)=f(J_\lambda(x))$, $J_\lambda x \to x$ as
$\lambda\to 0$, thanks to the dominated convergence theorem and
(\ref{eq:ori}) we can pass to the limit as $\lambda \to 0$ in the
equation
\[
u_\lambda(t) = e^{-tA}x - \int_0^t e^{-(t-s)A}f_\lambda(u_\lambda(s))\,ds
+ \eta \int_0^t e^{-(t-s)A} u_\lambda(s)\,ds + G_A(t),
\]
thus showing that $u$ is a mild solution of (\ref{eq:raga}).

The estimate for $\E\sup_{t\leq T}|u(t)|_{2d}^{2d}$ is an immediate
consequence of (\ref{eq:apri}).

We shall now prove uniqueness. In order to simplify notation a little,
we shall assume that $f$ is $\eta$-accretive, i.e. that $r \mapsto
f(r)+\eta r$ is accretive, and consequently we shall drop the first
term on the right hand side of (\ref{eq:raga}). This is of course
completely equivalent to the original setting. Let $\{e_k\}_{k \in
  \mathbb{N}} \subset D(A^*)$ be an orthonormal basis of $H$ and
$\varepsilon > 0$.  Denoting two solutions of (\ref{eq:raga}) by $u$
and $v$, we have
\begin{align*}
\bip{(I+\varepsilon A^*)^{-1}e_k}{u(t)-v(t)} =&
-\int_0^t \bip{A^* (I+\varepsilon A^*)^{-1}e_k}{u(s)-v(s)}ds\\
& - \int_0^t \bip{(I+\varepsilon A^*)^{-1}e_k}{f(u(s))-f(v(s))}ds
\end{align*}
for all $k \in \mathbb{N}$. Therefore, by It\^o's formula,
\begin{align*}
&\bip{(I+\varepsilon A^*)^{-1}e_k}{u(t)-v(t)}^2\\
&\qquad = - 2 \int_0^t \bip{A^*(I+\varepsilon A^*)^{-1}e_k}{u(s)-v(s)}
\bip{(I+\varepsilon A^*)^{-1}e_k}{u(s)-v(s)}ds\\
&\qquad\quad -2 \int_0^t
\bip{(I+\varepsilon A^*)^{-1}e_k}{f(u(s))-f(v(s))}
\bip{(I+\varepsilon A^*)^{-1}e_k}{u(s)-v(s)}ds.
\end{align*}
Summing over $k$ and recalling that $(I+\varepsilon A^*)^{-1}
=\big((I+\varepsilon A)^{-1}\big)^*$, we obtain
\begin{align*}
&\left| (I+\varepsilon A)^{-1} (u(t) - v(t)) \right|^2\\
&\qquad = -2 \int_0^t \bip{A (I+\varepsilon A)^{-1} (u(s) - v(s))}%
{(I+\varepsilon A)^{-1} (u(s) - v(s))}ds\\
&\qquad\quad -2\int_0^t\bip{(I+\varepsilon A)^{-1}(f(u(s))-f(v(s)))}%
{(I+\varepsilon A)^{-1} (u(s) - v(s))}ds.
\end{align*}
Using the monotonicity of $A$ and then letting $\varepsilon$ tend to
zero and we are left with
\begin{align*}
  |u(t) - v(t)|^2 &\leq
  -2 \int_0^t \bip{f(u(s))-f(v(s))}{u(s)-v(s)}ds\\
  &\leq 2\eta \int_0^t |u(s)-v(s)|^2\,ds,
\end{align*}
which immediately implies that $u=v$ by Gronwall's inequality.

Let us now prove Lipschitz continuity of the solution map. Set
$u^1:=u(x_1)$, $u^2:=u(x_2)$, and denote the strong solution of
(\ref{eq:raga}) with $A$ replaced by $A_\beta$, $f$ replaced by
$f_\lambda$, and initial condition $x_i$, $i=1,2$, by
$u^i_{\lambda\beta}$, $i=1,2$, respectively.
Then we have, omitting the dependence on time for
simplicity,
\[
(u^1_{\lambda\beta}-u^2_{\lambda\beta})' 
+ A_\beta(u^1_{\lambda\beta}-u^2_{\lambda\beta}) 
+ f_\lambda(u^1_{\lambda\beta})-f_\lambda(u^2_{\lambda\beta}) 
= \eta(u^1_{\lambda\beta}-u^2_{\lambda\beta})
\]
$\P$-a.s. in the strong sense. Multiplying, in the sense of the scalar
product of $L_2(D)$, both sides by
$u^1_{\lambda\beta}-u^2_{\lambda\beta}$ and taking into account the
monotonicity of $A$ and $f$, we get
\[
\frac12 |u^1_{\lambda\beta}(t)-u^2_{\lambda\beta}(t)|^2 
\leq |x_1-x_2|^2 
+ \eta \int_0^t |u^1_{\lambda\beta}(s)-u^2_{\lambda\beta}(s)|^2\,ds,
\]
which implies, by Gronwall's inequality and obvious estimates,
\[
\E \sup_{t \leq T} |u^1_{\lambda\beta}(t)-u^2_{\lambda\beta}(t)|^2 
\leq e^{2\eta T} \E|x_1-x_2|^2.
\]
Since, as seen above, $u^i_{\lambda\beta} \to u^i$, $i=1,2$, in
$\mathbb{H}_2(T)$ as $\beta \to 0$, $\lambda \to 0$, we conclude by
the dominated convergence theorem that $\|u^1-u^2\|_2 \leq e^{\eta T}
|x_1-x_2|_{\mathbb{L}_2}$.
\end{proof}
\begin{rmk}
  We would like to emphasize that proving uniqueness treating mild
  solutions as if they were strong solutions, as is very often done in
  the literature, does not appear to have a clear justification,
  unless the nonlinearity is Lipschitz continuous. In fact, if $u$ is
  a mild solution of (\ref{eq:raga}) and $u_\beta$ is a mild (or even
  strong) solution of the equation obtained by replacing $A$ with
  $A_\beta$ in (\ref{eq:raga}), one would at least need to know that
  $u_\beta$ converges to the given solution $u$, which is not clear at
  all and essentially equivalent to what one wants to prove, namely
  uniqueness.

  A general proof of uniqueness for mild solutions of stochastic
  evolution equations with dissipative nonlinear drift and
  multiplicative (Wiener and Poisson) noise is given in
  \cite{cm:IDAQP10}.
\end{rmk}
In order to establish well-posedness in the generalized mild sense, we
need the following a priori estimates, which are based on It\^o's
formula for the square of the norm and regularizations.
\begin{lemma}     \label{lm:cariddi}
  Let $x_1$, $x_2 \in \mathbb{L}_{2d}$, $G_1$, $G_2 \in
  \mathcal{L}_{d^*}$, and $u^1$, $u^2$ be mild solutions of
  (\ref{eq:raga}) with $x=x_1$, $G=G_1$ and $x=x_2$, $G=G_2$,
  respectively. Then one has
  \begin{equation}     \label{eq:tatu}
  e^{-2\eta t} \E|u^1(t)-u^2(t)|^2 \leq \E|x_1-x_2|^2 + \E\int_{Z_t}
  |G_1(s,z)-G_2(s,z)|^2\,m(dz)\,ds
  \end{equation}
  and
  \begin{equation}     \label{eq:tatu2}
  \E \sup_{t \leq T} |u^1(t)-u^2(t)|^2 \lesssim_T
  \E|x_1-x_2|^2 + \E\int_{Z_T} |G_1(s,z)-G_2(s,z)|^2\,m(dz)\,ds.
  \end{equation}
\end{lemma}
\begin{proof}
  Let $u_\lambda$ and $u_{\lambda\beta}$ be defined as in the proof of
  Proposition \ref{prop:tano}. Set $w^i(t)=e^{-\eta t}u^i_{\lambda\beta}(t)$.
  It\^o's formula for the square of the norm in $H$ yields
  \[
  |w^1(t)-w^2(t)|^2 = 2 \int_0^t \ip{w^1(s-)-w^2(s-)}{dw^1(s)-dw^2(s)} 
  + [w^1-w^2](t),
  \]
  i.e.
  \begin{align*}
  e^{-2\eta t}|u^1_{\lambda\beta}(t)-u^2_{\lambda\beta}(t)|^2
  &+ 2 \int_0^t e^{-2\eta s} \ip{A_\beta(u^1_{\lambda\beta}(s)-u^2_{\lambda\beta}(s))}{%
        u^1_{\lambda\beta}(s)-u^2_{\lambda\beta}(s)}\,ds\\
  &+ 2 \int_0^t e^{-2\eta s} \ip{f_\lambda(u^1_{\lambda\beta}(s))
                               -f_\lambda(u^2_{\lambda\beta}(s))}{%
        u^1_{\lambda\beta}(s)-u^2_{\lambda\beta}(s)}\,ds\\
  &\leq |x_1-x_2|^2 + [w^1-w^2](t) + M(t),
  \end{align*}
  where $M$ is a local martingale. In particular, since $A$ and $f$
  are monotone, we are left with
  \begin{equation}     \label{eq:leng}
  e^{-2\eta t}|u^1_{\lambda\beta}(t)-u^2_{\lambda\beta}(t)|^2
  \leq 
  |x_1-x_2|^2 + [w^1-w^2](t) + M(t).
  \end{equation}
  In particular, taking expectations on both sides (if necessary,
  along a sequence $\{\tau_n\}_{n\in\mathbb{N}}$ of localizing stopping
  times for the local martingale $M$, and then passing to the limit as
  $n \to \infty$), we obtain
  \[
  e^{-2\eta t} \E|u^1_{\lambda\beta}(t)-u^2_{\lambda\beta}(t)|^2
  \leq 
  \E|x_1-x_2|^2 + \E\int_{Z_t} |G_1(s,z)-G_2(s,z)|^2\,m(dz)\,ds,
  \]
  where we have used the identity
  \begin{equation}     \label{eq:leng2}
  \E[w^1-w^2](t) = \E\langle w^1-w^2 \rangle(t) =
  \E \int_{Z_t} e^{-2\eta s}|G_1(s,z)-G_2(s,z)|^2\,m(dz)\,ds.
  \end{equation}
  Recalling that $u^i_{\lambda\beta} \to u^i_\lambda$, $i=1,\,2$, in
  $\mathbb{H}_2(T)$ as $\beta$ go to zero (see the proof of Lemma
  \ref{lm:apri} or e.g. \cite[Prop.~3.11]{cm:JFA10}), we get that the
  above estimate holds true for $u^1_{\lambda}$, $u^2_{\lambda}$
  replacing $u^1_{\lambda\beta}$, $u^2_{\lambda\beta}$, respectively.
  Finally, since mild solutions are obtained as limits in
  $\mathbb{H}_2(T)$ of regularized solutions for $\lambda \to 0$,
  (\ref{eq:tatu}) follows.

  By (\ref{eq:leng}) and (\ref{eq:leng2}) we get
  \begin{align*}
  \E \sup_{t \leq T} e^{-2\eta t} |u^1_{\lambda\beta}(t)-u^2_{\lambda\beta}(t)|^2
  &\leq 
  \E|x_1-x_2|^2 + \E\int_{Z_T} |G_1(s,z)-G_2(s,z)|^2\,m(dz)\,ds\\
  &\quad + \E \sup_{t \leq T} |M(t)|.
  \end{align*}
  Note that
  \[
  M(t) = 2\int_{Z_t} \big\langle w^1_--w^2_-,(G_1(s,z)-G_2(s,z))\,\m(ds,dz)
  \big\rangle = 2 (w^1_--w^2_-) \cdot (X^1-X^2),
  \]
  where $X^i:=G_i \ast \m$, $i=1,\,2$. Thanks to Davis' and Young's
  inequalities we can write
  \begin{align*}
  \E \sup_{t \leq T} |M(t)| &\leq
        6 \E[(w^1_- - w^2_-) \cdot (X^1-X^2)](T)^{1/2}\\
  &\leq 6\E \big( \sup_{t \leq T} |w^1(t)-w^2(t)| \big) [X^1-X^2](T)^{1/2}\\
  &\leq 6\varepsilon \E \sup_{t \leq T} |w^1(t)-w^2(t)|^2 
  + 6\varepsilon^{-1} \E[X^1-X^2](T)\\
  &\leq 6\varepsilon \E \sup_{t \leq T} e^{-2\eta t}
      |u^1_{\lambda\beta}(t)-u^2_{\lambda\beta}(t)|^2\\
  &\quad + 6\varepsilon^{-1} \E\int_{Z_T} |G_1(s,z)-G_2(s,z)|^2\,m(dz)\,ds.
  \end{align*}
  Therefore we have
  \begin{multline*}
  (1-6\varepsilon) \E \sup_{t \leq T} e^{-2\eta t}
  |u^1_{\lambda\beta}(t)-u^2_{\lambda\beta}(t)|^2\\
  \leq \E|x_1-x_2|^2 +
  (1+6\varepsilon^{-1})\E\int_{Z_T} |G_1(s,z)-G_2(s,z)|^2\,m(dz)\,ds,
  \end{multline*}
  hence, passing to the limit as $\beta$ and $\lambda$ go to zero, we
  obtain (\ref{eq:tatu2}).
\end{proof}

\begin{prop}     \label{prop:dere}
  Assume that $x \in \mathbb{L}_2$ and $G \in \mathbb{L}_2(L_2(Z_T))$.
  Then (\ref{eq:raga}) admits a unique c\`adl\`ag generalized mild
  solution $u \in \mathbb{H}_2(T)$. Moreover, one has $x \mapsto u \in
  \dot{C}^{0,1}(\mathbb{L}_2,\mathbb{H}_2(T))$.
\end{prop}
\begin{proof}
  Let us choose a sequence $\{x_n\} \subset \mathbb{L}_{2d}$ such that
  $x_n \to x$ in $\mathbb{L}_2$, and a sequence $\{G_n\} \subset
  \mathcal{L}_{d^*}$ such that $G_n \to G$ in $\mathbb{L}_2(L_2(Z_T))$
  (e.g. by a cut-off procedure\footnote{For instance, one may
      set $G_n(\omega,t,z,x):=\mathbf{1}_{Z_n}(z) \big( (-n) \vee
      G(\omega,t,z,x) \wedge n \big)$, where
      $\{Z_n\}_{n\in\mathbb{N}}$ is an increasing sequence of subsets
      of $Z$ such that $Z_n \uparrow Z$ and $m(Z_n)<\infty$.}).  By
  Proposition \ref{prop:tano} the stochastic equation
  \[
  du + Au\,dt + f(u)\,dt = \eta u\,dt + G_n\,d\m, \qquad u(0)=x_n
  \]
  admits a unique mild solution $u_n$. Then (\ref{eq:tatu2})
  yields
  \[
  \E \sup_{t \leq T} |u_n(t)-u_m(t)|^2 \lesssim
  \E|x_n-x_m|^2 + \E \int_{Z_T} |G_n(s,z)-G_m(s,z)|^2\,m(dz)\,ds
  \]
  In particular $\{u_n\}$ is a Cauchy sequence in $\mathbb{H}_2(T)$,
  whose limit $u \in \mathbb{H}_2(T)$ is a generalized mild solution
  of (\ref{eq:raga}). Since $u_n$ is c\`adl\`ag for each $n$ by
  Proposition \ref{prop:tano}, $u$ is also c\`adl\`ag.

  Moreover, it is immediate that $x_i \mapsto u_i$, $i=1,2$, satisfies
  $\|u_1-u_2\|^2_2 \lesssim |x_1-x_2|^2_{\mathbb{L}_2}$, i.e. the
  solution map is Lipschitz, which in turn implies uniqueness of the
  generalized mild solution.
\end{proof}

\begin{rmk}
  One could also prove well-posedness in $\mathcal{H}_2(T)$, simply
  using estimate (\ref{eq:tatu}) instead of (\ref{eq:tatu2}). In this
  case one can also get explicit estimates for the Lipschitz constant
  of the solution map. On the other hand, one cannot conclude that a
  solution in $\mathcal{H}_2(T)$ is c\`adl\`ag, as the subset of
  c\`adl\`ag processes is not closed in $\mathcal{H}_2(T)$.
\end{rmk}

\subsection{Multiplicative noise}
Let us consider the stochastic evolution equation
\begin{equation}       \label{eq:culo}
  du(t) + Au(t)\,dt + f(u(t))\,dt = \eta u(t)\,dt
  + \int_Z G(t,z,u(t-))\,\m(dt,dz)
\end{equation}
with initial condition $u(0)=x$, where $G: \Omega \times [0,T] \times Z
\times \erre \times D \to \erre$ is a $\mathcal{P} \otimes \mathcal{Z}
\otimes \mathcal{B}(\erre) \otimes \mathcal{B}(\erre^n)$-measurable
function, and we denote its associated Nemitski operator, which is a
mapping from $\Omega \times [0,T] \times Z \times H \to H$, again by $G$.

We have the following well-posedness result for (\ref{eq:culo}) in the
generalized mild sense.

\begin{thm}     \label{thm:main}
  Assume that $x \in \mathbb{L}^2$ and $G$ satisfies the Lipschitz
  condition
  \[
  \E\int_Z |G(s,z,u)-G(s,z,v)|^2\,m(dz)\,ds \leq h(s)|u-v|^2,
  \]
  where $h \in L_1([0,T])$. Then (\ref{eq:culo}) admits a unique
  generalized solution $u \in \mathbb{H}_2(T)$. Moreover, the solution
  map is Lipschitz from $\mathbb{L}_2$ to $\mathbb{H}_2(T)$.
\end{thm}
\begin{proof}
  For $v \in \mathbb{H}_2(T)$ and c\`adl\`ag, consider the equation
  \begin{equation}       \label{eq:orto}
    du(t) + Au(t)\,dt + f(u(t))\,dt = \eta u(t)\,dt
    + \int_Z G(s,z,v(s-))\,\m(ds,dz),
    \qquad u(0)=x.
  \end{equation}
  Since $(s,z)\mapsto G(s,z,v(s-))$ satisfies the hypotheses of
  Proposition \ref{prop:dere}, (\ref{eq:orto}) admits a unique
  generalized mild solution belonging to $\mathbb{H}_2(T)$. Let us
  denote the map associating $v$ to $u$ by $F$. We are going to prove
  that $F$ is well-defined and is a contraction on
  $\mathbb{H}_{2,\alpha}(T)$ for a suitable choice of $\alpha>0$.
  Setting $u^i=F(v^i)$, $i=1,2$, with $v^1$, $v^2 \in
  \mathbb{H}_2(T)$, we have
  \begin{multline*}
  d(u^1-u^2) + [A(u^1-u^2) + f(u^1)-f(u^2)]\,dt \\
  = \eta(u^1-u^2)\,dt
  + \int_Z [G(\cdot,\cdot,v^1_-) - G(\cdot,\cdot,v^2_-)]\,d\m
  \end{multline*}
  in the mild sense, with obvious meaning of the (slightly simplified)
  notation. We are going to assume that $u^1$ and $u^2$ are strong
  solutions, without loss of generality: in fact, one otherwise
  approximate $A$, $f$ and $G$ with $A_\beta$, $f_\lambda$, and $G_n$,
  respectively, and passes to the limit in equation (\ref{eq:fondo})
  below, leaving the rest of argument unchanged. Setting
  $w^i(t)=e^{-\alpha t}u^i(t)$, $i=1,2$, we have, by an argument
  completely similar to the one used in the proof of Lemma
  \ref{lm:cariddi},
  \begin{multline*}
  |w^1(t)-w^2(t)|^2 \leq 
  (\eta-\alpha)\int_0^t e^{-2\alpha s} |u^1(s)-u^2(s)|^2\,ds + [w^1-w^2](t)\\
  + 2\int_{Z_t} \big\langle
  e^{-2\alpha s}(u^1(s-)-u^2(s-),(G(s,z,v^1(s-))-G(s,z,v^2(s-)))\,\m(ds,dz)
  \big\rangle.
  \end{multline*}
The previous inequality in turn implies
\begin{align*}
\|u^1-u^2\|^2_{2,\alpha} \leq\;&
(\eta-\alpha)\int_0^T \E\sup_{s\leq t}e^{-2\alpha s} |u^1(s)-u^2(s)|^2\,ds\\
& + 2 \E\sup_{t \leq T} \big| (w^1_--w^2_-) \cdot (X^1-X^2) \big|\\
& + \E \int_0^T\!\!\int_Z e^{-2\alpha s}
|G(s,z,v^1(s-))-G(s,z,v^2(s-))|^2\,m(dz)\,ds,
\end{align*}
where we have set $X^i:=G(\cdot,\cdot,v^i_-)\star \m$ and we have used
the identities
\begin{align*}
\E\sup_{t\leq T} [w^1-w^2](t) &= \E [w^1-w^2](T) \\
&= \E \int_0^T\!\!\!\int_Z e^{-2\alpha s}
      |G(s,z,v^1(s))-G(s,z,v^2(s))|^2\,m(dz)\,ds.
\end{align*}
An application of Davis' and Young's inequalities, as in the proof of
Lemma \ref{lm:cariddi}, yields
\begin{multline*}
2 \E\sup_{t \leq T} \big| (w^1_--w^2_-) \cdot (X^1-X^2) \big|
\leq 6\varepsilon \E\sup_{t \leq T} |w^1(t)-w^2(t)|^2\\
+ 6\varepsilon^{-1} \E \int_0^T\!\!
\int_Z e^{-2\alpha s} |G(s,z,v^1(s))-G(s,z,v^2(s))|^2\,m(dz)\,ds,
\end{multline*}
because $[X^1-X^2]=[w^1-w^2]$.
We have thus arrived at the estimate
\begin{multline}     \label{eq:fondo}
(1-6\varepsilon)\|u^1-u^2\|_{2,\alpha}^2 \leq
(\eta-\alpha)\int_0^T \E\sup_{s\leq t} e^{-2\alpha s}|u^1(s)-u^2(s)|^2\,dt\\
+ (1+6\varepsilon^{-1}) \E \int_0^T\!\!
\int_Z e^{-2\alpha s} |G(s,z,v^1(s))-G(s,z,v^2(s))|^2\,m(dz)\,ds
\end{multline}
Setting $\varepsilon=1/12$ and
$\phi(t) = \E \sup_{s\leq t} e^{-2\alpha s}|u^1(s)-u^2(s)|^2$,
we can write, by the hypothesis on $G$,
\[
\phi(T) \leq 2(\eta-\alpha)\int_0^T \phi(t)\,dt 
+ 146|h|_{L_1} \|v^1-v^2\|^2_{2,\alpha},
\]
hence, by Gronwall's inequality,
\[
\|u^1-u^2\|^2_{2,\alpha} = \phi(T) \leq 146 |h|_1 e^{2(\eta-\alpha)T}
\|v^1-v^2\|^2_{2,\alpha}.
\]
Choosing $\alpha$ large enough, we obtain that there exists a constant
$N=N(T)<1$ such that $\|F(v^1)-F(v^2)\|_{2,\alpha} \leq N
\|v^1-v^2\|_{2,\alpha}$. Banach's fixed point theorem then implies
that $F$ admits a unique fixed point in $\mathbb{H}_{2,\alpha}(T)$,
which is the (unique) generalized solution of (\ref{eq:culo}),
recalling that the norms $\|\cdot\|_{2,\alpha}$, $\alpha\geq 0$, are
all equivalent. Since the fixed point of $F$ can also be obtained as a
limit of c\`adl\`ag processes in $\mathbb{H}_2(T)$, by the well-known
method of Picard's iterations, we also infer that the generalized mild
solution is c\`adl\`ag.

Moreover, denoting $u(x_1)$ and $u(x_2)$ by $u^1$ and $u^2$
respectively, an argument similar to the one leading to
(\ref{eq:fondo}) yields the estimate
\[
\psi(T) \leq \E|x_1-x_2|^2 + 2(\eta-\alpha)\int_0^T \psi(t)\,dt
+ 146 \int_0^T h(t) \psi(t)\,dt,
\]
where $\psi(t):=\E\sup_{s\leq t} |u^1(s)-u^2(s)|^2$. By Gronwall's
inequality we get
\[
\| u^1 - u^2 \|_{2,\alpha}^2 \leq e^{2(\eta-\alpha) + 146|h|_{L_1}}
|x_1-x_2|_{\mathbb{L}_2}^2,
\]
which proves that $x \mapsto u(x)$ is Lipschitz from $\mathbb{L}_2$ to
$\mathbb{H}_{2,\alpha}(T)$, hence also from $\mathbb{L}_2$ to
$\mathbb{H}_2(T)$ by the equivalence of the norms
$\|\cdot\|_{2,\alpha}$.
\end{proof}

\begin{rmk}     \label{rmk:ext}
  As briefly mentioned in the introduction, one may prove (under
  suitable assumptions) global well posedness for stochastic evolution
  equations obtained by adding to the right-hand side of
  (\ref{eq:culo}) a term of the type $B(t,u(t))\,dW(t)$, where $W$ is
  a cylindrical Wiener process on $L_2(D)$, and $B$ satisfies a
  Lipschitz condition analogous to the one satisfied by $G$ in Theorem
  \ref{thm:main}. An inspection of our proof reveals that all is
  needed is a maximal estimate of the type (\ref{eq:BJconv}) for
  stochastic convolutions driven by Wiener processes. To this purpose
  one may use, for instance, \cite[Thm.~2.13]{DP-K}. Let us also
  remark that many sophisticated estimates exist for stochastic
  convolutions driven by Wiener processes. Full details on well
  posedness as well as existence and uniqueness of invariant measures
  for stochastic evolution equations of (quasi)dissipative type driven
  by both multiplicative Poisson and Wiener noise will be given in a
  forthcoming article.
\end{rmk}

\section{Invariant measures and Ergodicity}     \label{sec:ergo}
Throughout this section we shall additionally assume that $G: Z \times
H \to H$ is a (deterministic)
$\mathcal{Z}\otimes\mathcal{B}(H)$-measurable function satisfying the
Lipschitz assumption
\[
\int_Z |G(z,u)-G(z,v)|^2\,m(dz) \leq K|u-v|^2,
\]
for some $K>0$. The latter assumption guarantees that the evolution
equation is well-posed by Theorem \ref{thm:main}. Moreover, it is easy
to see that the solution is Markovian, hence it generates a semigroup
via the usual formula $P_t\varphi(x):=\E\varphi(u(t,x))$, $\varphi \in
B_b(H)$. Here $B_b(H)$ stands for the set of bounded Borel functions
from $H$ to $\erre$.

\subsection{Strongly dissipative case}
Throughout this subsection we shall assume that there
exist $\beta_0$ and $\omega_1 > K$ such that
\begin{equation}     \label{eq:spisa}
  2\ip{A_\beta u-A_\beta v}{u-v}
+ 2\ip{f_\lambda(u)-f_\lambda(v)}{u-v} - 2\eta|u-v|^2 
  \geq \omega_1|u-v|^2
\end{equation}
for all $\beta \in ]0,\beta_0[$, $\lambda \in ]0,\beta_0[$,
and for all $u$, $v \in H$.
This is enough to guarantee existence and uniqueness of an ergodic
invariant measure for $P_t$, with exponentially fast convergence to
equilibrium.
\begin{prop}     \label{prop:ergof}
  Under hypothesis (\ref{eq:spisa}) there exists a unique invariant
  measure $\nu$ for $P_t$, which satisfies the following properties:
  \begin{itemize}
  \item[(i)] $\displaystyle \int |x|^2\,\nu(dx) < \infty$;
  \item[(ii)] let $\varphi \in \dot{C}^{0,1}(H,\erre)$ and $\lambda_0
    \in \mathcal{M}_1(H)$. Then one has
  \[
  \Big| \int_H P_t\varphi(x)\,\lambda_0(dx)
       -\int_H\varphi\,\nu(dy) \Big| \leq [\varphi]_1 e^{-\omega_1 t} 
               \int_{H\times H} |x-y|\,\lambda_0(dx)\,\nu(dy)
  \]  
  \end{itemize}
\end{prop}
Following a classical procedure (see
e.g. \cite{DZ96,PZ-libro,PreRoeck}), let us consider the stochastic
equation
\begin{equation}     \label{eq:strano}
du(t) + (Au(t) + f(u))\,dt = \eta u(t)\,dt 
+ \int_Z G(z,u(t-))\,d\m_1(dt,dz),
\qquad u(s)=x,
\end{equation}
where $s \in ]-\infty,t[$, $\m_1=\mu_1-\mathrm{Leb}\otimes m$, and
\[
\mu_1(t,B) = \begin{cases}
\mu(t,B), & t \geq 0,\\
\mu_0(-t,B), & t < 0,
\end{cases}
\]
for all $B \in \mathcal{Z}$, with $\mu_0$ an independent copy of
$\mu$. The filtration $(\bar{\mathcal{F}}_t)_{t\in\erre}$ on which
$\mu_1$ is considered can be constructed as follows:
\[
\bar{\mathcal{F}}_t := \bigcap_{s>t} \bar{\mathcal{F}}^0_s,
\qquad
\bar{\mathcal{F}}^0_s := \sigma\big( \{\mu_1([r_1,r_2],B):\;
-\infty < r_1 \leq r_2 \leq s,\; B \in \mathcal{Z}\},\mathcal{N}\big),
\]
where $\mathcal{N}$ stands for the null sets of the probability space
$(\Omega,\mathcal{F},\P)$.
We shall denote the value at time $t \geq
s$ of the solution of (\ref{eq:strano}) by $u(t;s,x)$.

For the proof of Proposition \ref{prop:ergof} we need the following lemma.
\begin{lemma}     \label{lm:vectr}
  There exists a random variable $\zeta \in \mathbb{L}_2$ such that
  $u(0;s,x) \to \zeta$ in $\mathbb{L}_2$ as $s \to -\infty$ for all $x
  \in \mathbb{L}_2$. Moreover, there exists a constant $N$ such that
  \begin{equation}     \label{eq:velo}
    \E|u(0;s,x)-x|^2 \leq e^{-2\omega_1 |s|} N (1 + \E|x|^2)
  \end{equation}
  for all $s<0$.
\end{lemma}
\begin{proof}
  Let $u$ be the generalized mild solution of (\ref{eq:strano}).
  Define $\Gamma(t,z):=G(z,u(t-))$, and let $\Gamma_n$ be an
  approximation of $\Gamma$, as in the proof of Proposition
  \ref{prop:dere}. Let us denote the strong solution of the equation
  \[
  du(t) + (A_\beta u(t) + f_\lambda(u(t)))\,dt = \eta u(t)\,dt 
  + \int_Z \Gamma_n(t,z)\,d\m_1(dt,dz),
  \qquad u(s)=x,
  \]
  by $u_{\lambda\beta}^n$.  By It\^o's lemma we can write
  \begin{multline}     \label{eq:toh}
  |u_{\lambda\beta}^n(t)|^2
  + 2\int_s^t \big[\ip{A_\beta u_{\lambda\beta}^n(r)}{u_{\lambda\beta}^n(r)}
        + \ip{f_\lambda(u_{\lambda\beta}^n(r))}{u_{\lambda\beta}^n(r)}
        -\eta |u_{\lambda\beta}^n(r)|^2\big]\,dr \\
  = |x|^2 + 2\int_s^t\!\!\int_Z
          \ip{\Gamma_n(r,z)}{u_{\lambda\beta}^n(r)}\,\m_1(dr,dz)
    + \int_s^t\!\!\int_Z |\Gamma_n(r,z)|^2\,\mu_1(dr,dz).     
  \end{multline}
  Note that we have, by Young's inequality, for any $\varepsilon>0$,
  \begin{align*}
    -\ip{f_\lambda(u)}{u} &= -\ip{f_\lambda(u)-f(0)}{u-0}-\ip{f_\lambda(0)}{u}\\
    &\leq -\ip{f_\lambda(u)-f(0)}{u-0} + \frac{\varepsilon}{2}|u|^2 +
    \frac{1}{2\varepsilon}|f_\lambda(0)|^2.
  \end{align*}
  Since $f_\lambda(0) \to f(0)$ as $\lambda \to 0$, there exists
  $\delta>0$, $\lambda_0>0$ such that
  \[
  |f_\lambda(0)|^2 \leq |f(0)|^2 + \delta/2
  \qquad \forall \lambda<\lambda_0.
  \]
  By (\ref{eq:spisa}) we thus have, for $\beta<\beta_0$, $\lambda <
  \lambda_0 \wedge \beta_0$,
  \begin{align*}
    &-2\ip{A_\beta u_{\lambda\beta}^n}{u_{\lambda\beta}^n}
     - 2\ip{f_\lambda(u_{\lambda\beta}^n)}{u_{\lambda\beta}^n}
     + 2\eta|u_{\lambda\beta}^n|^2\\
    &\qquad \leq -\omega_1 |u_{\lambda\beta}^n|^2
     + \varepsilon |u_{\lambda\beta}^n|^2 +
     \varepsilon^{-1} |f(0)|^2 + \delta.
  \end{align*}
  Taking expectations in (\ref{eq:toh}), applying the above
  inequality, and passing to the limit as $\beta \to 0$, $\lambda \to
  0$, and $n \to \infty$, yields
  \begin{align*}
  \E |u(t)|^2 &\leq \E|x|^2 
  - (\omega_1-\varepsilon)\int_s^t \E|u(r)|^2\,dr
  + \big( \varepsilon^{-1}|f(0)|^2 + \delta \big)(t-s)\\
  &\quad + \E\int_s^t\!\!\int_Z |\Gamma(r,z)|^2\,m(dz)\,dr.
  \end{align*}
  Note that, similarly as before, we have
  \[
  \int_Z |G(u,z)|^2\,m(dz) \leq (1+\varepsilon) K |u|^2
  + (1+\varepsilon^{-1}) \int_Z |G(0,z)|^2\,m(dz)
  \]
  for any $u \in H$, therefore, by definition of $\Gamma$,
  \[
  \E\int_s^t\!\!\int_Z |\Gamma(r,z)|^2\,m(dz)\,dr \leq
  (1+\varepsilon) K \int_s^t \E|u(r)|^2\,dr
  + (1+\varepsilon^{-1}) \big| G(0,\cdot) \big|^2_{L_2(Z,m)} (t-s).
  \]
  Setting
  \[
  \omega_2 := \omega_1 - K - \varepsilon(1+K),
  \qquad
  N := \varepsilon^{-1} |f(0)|^2 + \delta
  + (1+\varepsilon^{-1}) \big| G(0,\cdot) \big|^2_{L_2(Z,m)},
  \]
  we are left with
  \[
  \E |u(t)|^2 \leq \E|x|^2 - \omega_2\int_s^t \E|u(r)|^2\,dr + N(t-s).
  \]
  We can now choose $\varepsilon$ such that $\omega_2>0$. Gronwall's
  inequality then yields
  \begin{equation}     \label{eq:*}
  \E |u(t)|^2 \lesssim 1 + e^{-\omega_2(t+|s|)} \E|x|^2.
  \end{equation}
  Set $u_1(t):=u(t;s_1,x)$, $u_2(t):=u(t;s_2,x)$ and
  $w(t)=u_1(t)-u_2(t)$, with $s_2<s_1$. Then $w$ satisfies the
  equation
  \[
  dw + Aw\,dt + (f(u_1)-f(u_2))\,dt = \eta w\,dt
  + (G(u_1)-G(u_2))\,d\m,
  \]
  with initial condition $w(s_1)=x-u_2(s_1)$, in the generalized mild
  sense.  By an argument completely similar to the above one, based on
  regularizations, It\^o's formula, and passage to the limit, we
  obtain
  \[
  \E|w(t)|^2 \leq \E|x-u_2(s_1)|^2 - (\omega_1-K) \int_{s_1}^t
  \E|w(r)|^2\,dr,
  \]
  and hence, by Gronwall's inequality,
  \[
  \E|u_1(0)-u_2(0)|^2 = \E|w(0)|^2 \leq e^{-(\omega_1-K)|s_1|}
  \E|x-u_2(s_1)|^2.
  \]
  Estimate (\ref{eq:*}) therefore implies that there exists a constant
  $N$ such that
  \begin{equation}    \label{eq:**}
    \E|u_1(0)-u_2(0)|^2 \leq
    e^{-(\omega_1-K)|s_1|} N(1+\E|x|^2),
  \end{equation}
  which converges to zero as $s_1 \to -\infty$. We have thus proved
  that $\{u(0;s,x)\}_{s\leq 0}$ is a Cauchy net in $\mathbb{L}_2$,
  hence there exists $\zeta=\zeta(x) \in \mathbb{L}_2$ such that
  $u(0;s,x) \to \zeta$ in $\mathbb{L}_2$ as $s \to -\infty$.
  Let us show that $\zeta$ does not depend on $x$. In fact, let $x$,
  $y \in \mathbb{L}_2$ and set $u_1(t)=u(t;s,x)$, $u_2(t)=u(t;s,y)$.
  Yet another argument based on approximations, It\^o's formula for
  the square of the norm and the monotonicity assumption
  (\ref{eq:spisa}) yields, in analogy to a previous computation,
  \begin{equation}     \label{eq:dedrio}
    \E|u_1(0)-u_2(0)|^2 \leq e^{-(\omega_1-K)|s|}\E|x-y|^2,
  \end{equation}
  which implies $\zeta(x)=\zeta(y)$, whence the claim. Finally,
  (\ref{eq:**}) immediately yields (\ref{eq:velo}).
\end{proof}
\begin{proof}[Proof of Proposition \ref{prop:ergof}]
  Let $\nu$ be the law of the random variable $\zeta$ constructed in
  the previous lemma. Since $\zeta \in \mathbb{L}_2$, (i) will follow
  immediately once we have proved that $\nu$ is invariant for $P_t$.
  The invariance and the uniqueness of $\nu$ is a well-known
  consequence of the previous lemma, see e.g. \cite{DP-K}.

Let us prove (ii): we have
\begin{align*}
  &\Big| \int_H P_t\varphi(x)\,\lambda_0(dx) 
        - \int_H \varphi(y)\,\nu(dy) \Big|\\
  &\qquad = \Big| \int_H\!\!\int_H P_t\varphi(x)\,\lambda_0(dx)\,\nu(dy)
   - \int_H\!\!\int_H P_t\varphi(y)\,\lambda_0(dx)\,\nu(dy) \Big|\\
  &\qquad \leq \int_{H \times H} |P_t\varphi(x)
                -P_t\varphi(y)|\,\lambda_0(dx)\,\nu(dy)\\
  &\qquad \leq [\varphi]_1 e^{-\omega_1 t} 
               \int_{H\times H} |x-y|\,\lambda_0(dx)\,\nu(dy),
\end{align*}
where in the last step we have used the estimate (\ref{eq:dedrio}).
\end{proof}

\subsection{Weakly dissipative case}
In this subsection we replace the strong dissipativity condition
(\ref{eq:spisa}) with a super-linearity assumption on the nonlinearity
$f$, and we prove existence of an invariant measure by an argument
based on Krylov-Bogoliubov's theorem.

We assume that $-A$ satisfies the weak sector condition and let
$(\EE,D(\EE))$ be the associated closed coercive form (see
\cite[{\S}I.2]{MR}). We set $\mathcal{H}:=D(\EE)$, endowed with the
norm associated to the inner product
$\EE_1(\cdot,\cdot):=\EE(\cdot,\cdot)+\ip{\cdot}{\cdot}$.
\begin{thm}     \label{thm:ultimo}
  Assume that
  \begin{itemize}
  \item[(i)] $f$ satisfies the super-linearity condition $\ip{f(r)}{r}
    \geq b|r|^{2(1+\alpha)}/2$, $b>0$, $\alpha > 0$.
  \item[(ii)] $\mathcal{H}$ is compactly embedded into $L_2(D)$. 
  \end{itemize}
  Then there exists an invariant measure for the transition semigroup
  associated to the generalized mild solution of (\ref{eq:culo}).
\end{thm}
\begin{proof}
  Let $u$, $u_{\lambda\beta}^n$, $\Gamma$, and $\Gamma_n$ be defined
  as in the proof of Lemma \ref{lm:vectr}. Then an application of
  It\^o's formula yields the estimate
  \begin{multline}
  \E|u_{\lambda\beta}^n(t)|^2
  + 2\E\int_0^t \big[ \ip{A_\beta u_{\lambda\beta}^n(s)}{u_{\lambda\beta}^n(s)} 
  + \ip{f_\lambda(u_{\lambda\beta}^n(s))}{u_{\lambda\beta}^n(s)} \big]\,ds\\
  \leq \E|x|^2 + \E \int_0^t \big[ 2\eta |u_{\lambda\beta}^n(s)|^2
  + |\Gamma_n(s,\cdot)|^2_{L_2(Z,m)}\big]\,ds.
  \label{eq:mistica}
  \end{multline}
  Since
  \[
  \ip{f_\lambda(r)}{r} = \ip{f(J_\lambda r)}{J_\lambda r + (r-J_\lambda r)}
  = \ip{f(J_\lambda r)}{J_\lambda r} + \lambda |f_\lambda(r)|^2,
  \]
  we obtain, taking into account the monotonicity of $A_\beta$,
  \begin{align*}
  \E|u_{\lambda\beta}^n(t)|^2
  &\leq \E |x|^2 + 2\eta \int_0^t \E|u_{\lambda\beta}^n(s)|^2\,ds
  - 2\E \int_0^t
  \ip{f(J_\lambda u_{\lambda\beta}^n(s))}{J_\lambda u_{\lambda\beta}^n(s)}\,ds\\
  &\quad + \E\int_0^t |\Gamma_n(s,\cdot)|^2_{L_2(Z,m)}\,ds.
  \end{align*}
  By assumption (i) and Jensen's inequality, we have
  \begin{align*}
  - 2\int_0^t
    \E \ip{f(J_\lambda u_{\lambda\beta}^n(s))}{J_\lambda u_{\lambda\beta}^n(s)}\,ds
  &\leq -b\int_0^t \E|J_\lambda u_{\lambda\beta}^n(s)|^{2+2\alpha}\,ds\\
  &\leq -b\int_0^t \big(\E|J_\lambda u_{\lambda\beta}^n(s)|^2\big)^{1+\alpha}\,ds,
  \end{align*}
  thus also
  \begin{align*}
  \E|u_{\lambda\beta}^n(t)|^2
  &\leq \E |x|^2 + 2\eta\int_0^t \E|u_{\lambda\beta}^n(s)|^2\,ds
  - b\int_0^t \big(\E|J_\lambda u_{\lambda\beta}^n(s)|^2\big)^{1+\alpha}\,ds\\
  &\quad + \E\int_0^t |\Gamma_n(s,\cdot)|^2_{L_2(Z,m)}\,ds.
  \end{align*}
  Passing to the limit as $\beta \to 0$, $\lambda \to 0$, $n \to
  \infty$, recalling the definition of $\Gamma$ and $\Gamma_n$, and
  taking into account the Lipschitz continuity of $G$, shows that
  $y(t):=\E|u(t)|^2$ satisfies the differential inequality (in its
  integral formulation, to be more precise)
  \[
  y' \leq ay - by^{1+\alpha} + c, \qquad y(0)=\E|x|^2,
  \]
  for some positive constants $a$ and $c$. By simple ODE techniques
  one obtains that $y(t)$ is bounded for all $t$, i.e. $\E|u(t)|^2
  \leq C$ for all $t \geq 0$, for some positive constant $C$.

  Taking into account the monotonicity of $f_\lambda$,
  (\ref{eq:mistica}) also implies
  \begin{align*}
  \E\int_0^t \ip{A_\beta u_{\lambda\beta}^n(s)}{u_{\lambda\beta}^n(s)}\,ds
  &\leq \E |x|^2 + 2\eta\int_0^t \E|u_{\lambda\beta}^n(s)|^2\,ds
  - b\int_0^t \big(\E|J_\lambda u_{\lambda\beta}^n(s)|^2\big)^{1+\alpha}\,ds\\
  &\quad + \E\int_0^t |\Gamma_n(s,\cdot)|^2_{L_2(Z,m)}\,ds.
  \end{align*}
  As we have seen above, the right-hand side of the inequality
  converges, as $\beta \to 0$, $\lambda \to 0$, $n \to \infty$, to
  \[
  \E |x|^2 + N\int_0^t \big(1+\E|u(s)|^2\big)\,ds
  - b\int_0^t \big(\E|u(s)|^2\big)^{1+\alpha}\,ds
  \lesssim 1 + t,
  \]
  where $N$ is a constant that does not depend on $\lambda$, $\beta$,
  and $n$, and we have used the fact that $\E|u(t)|^2$ is bounded for
  all $t \geq 0$. In analogy to an earlier argument, setting $z_\beta
  := (I+\beta A)^{-1}z$, $z \in H$, we have
  \[
  \ip{A_\beta z}{z} =
  \ip{Az_\beta}{z_\beta+z-z_\beta}
  = \ip{Az_\beta}{z_\beta} + \beta |A_\beta z|^2.
  \]
  In particular, setting $v_{\lambda\beta}^n := (I+\beta A)^{-1}
  u_{\lambda\beta}^n \in D(\EE)$, we obtain
  \[
  \E \int_0^t \EE\big(v_{\lambda\beta}^n(s),v_{\lambda\beta}^n(s)\big)\,ds
  \lesssim 1+t
  \]
  for small enough $\beta$, $\lambda$, and $1/n$.  Since also
  $v_{\lambda\beta}^n(s) \to u(s)$ in $\mathbb{H}_2(T)$, it follows
  that $u \in L_2(\Omega\times[0,t],D(\EE))$ and $v_{\lambda\beta}^n
  \to u$ weakly in $L_2(\Omega\times[0,t],D(\EE))$, where $D(\EE)$ is
  equipped with the norm $\EE_1^{1/2}(\cdot,\cdot)$, and
  \[
  \E\int_0^t \EE_1(u(s),u(s))\,ds \lesssim 1+t.
  \]
  Let us now define the sequence of probability measures
  $(\nu_n)_{n\geq 1}$ on the Borel set of $H=L_2(D)$ by
  \[
  \int_H \phi\,d\nu_n = \frac1n \int_0^n \E\phi(u(s,0))\,ds,
  \qquad \phi \in B_b(H).
  \]
  Then
  \[
  \int |x|^2_{\mathcal{H}}\,\nu_n(dx)
  = \frac1n \int_0^n \E\,\EE_1(u(s,0),u(s,0)\,ds \lesssim 1,
  \]
  thus also, by Markov's inequality,
  \[
  \sup_{n \geq 1} \nu_n(B^c_R) \lesssim \frac1R \xrightarrow{R\to\infty} 0,
  \]
  where $B_R^c$ stands for the complement in $\mathcal{H}$ of the
  closed ball of radius $R$ in $\mathcal{H}$. Since balls in
  $\mathcal{H}$ are compact sets of $L_2(D)$, we infer that
  $(\nu_n)_{n \geq 1}$ is tight, and Krylov-Bogoliubov's theorem
  guarantees the existence of an invariant measure.
\end{proof}

\appendix

\section{Auxiliary results}
The following proposition is a slight modification of
\cite[Thm.~6.1.2]{Pazy} and it is used in the proof of Lemma
\ref{lm:apri}. Here $[0,T] \subset \erre$ and $E$ is a separable
Banach space.

\begin{prop}     \label{prop:pazzo}
  Assume that $f:[0,T] \times E \to E$ satisfies
  \[
  |f(t,x)-f(t,y)| \leq N|x-y|, \qquad \forall t \in [0,T], \; x,\,y \in E,
  \]
  where $N$ is a constant independent of $t$, and there exists $a \in
  E$ such that $t \mapsto f(t,a) \in L_1([0,T];E)$. If $A$ is the
  infinitesimal generator of a strongly continuous semigroup $e^{tA}$
  on $E$ and $u_0 \in E$, then the integral equation
  \begin{equation}     \label{eq:ecchecc}
  u(t) = e^{tA}u_0 + \int_0^t e^{(t-s)A} f(s,u(s))\,ds,
  \qquad t \in [0,T],
  \end{equation}
  admits a unique solution $u \in C([0,T],E)$.
\end{prop}
\begin{proof}
  As a first step, let us show that, if $v \in L_\infty([0,T];E)$,
  then $t \mapsto f(t,v(t)) \in L_1([0,T];E)$. In fact, we have
  \begin{align*}
    |f(t,v(t))| &\leq |f(t,v(t))-f(t,a)| + |f(t,a)|\\
    &\leq N|v(t)-a| + |f(t,a)| \leq N|v(t)|+N|a|+|f(t,a)|,
  \end{align*}
  thus also
  \[
  \int_0^T |f(t,v(t))|\,dt \leq NT|a| + N T|v|_{L_\infty}
  + |f(\cdot,a)|_{L_1} < \infty.
  \]
  As a second step, we show that the map
  \[
  [\mathfrak{F}v](t) := e^{tA}u_0 + \int_0^t e^{(t-s)A} f(s,v(s))\,ds
  \]
  is a (local) contraction in $L_\infty([0,T];E)$. In fact, setting
  $M_T = \sup_{t \in [0,T]} |e^{tA}|$, we have
  \[
  \sup_{t \in [0,T]} \big| [\mathfrak{F}(v)](t) \big| \leq
  M_T|u_0| + M_T \int_0^T |f(s,v(s)|\,ds < \infty,
  \]
  because $f(\cdot,v(\cdot)) \in L_1([0,T];E)$, as proved above. We
  also have
  \begin{align*}
    \sup_{t \in [0,T]} \big| [\mathfrak{F}v](t) - [\mathfrak{F}w](t) \big|
    &\leq N M_T \sup_{t \in [0,T]} \int_0^t |v(s)-w(s)|\,ds\\
    &\leq N M_T T |v-w|_{L_\infty},
  \end{align*}
  so that $N M_T T_0 < 1$ for $T_0$ small enough. Then $\mathfrak{F}$
  admits a unique fixed point in $L_\infty([0,T_0];E)$, and by a
  classical patching argument we obtain the existence of a unique
  solution $u \in L_\infty([0,T];E)$ to the integral equation
  (\ref{eq:ecchecc}). As a last step, it remains to prove that $u \in
  C([0,T];E)$. To this purpose, it suffices to show that $g \in
  L_1([0,T];E)$ implies $F \in C([0,T];E)$, with
  \[
  F(t) := \int_0^t e^{(t-s)A}g(s)\,ds.
  \]
  In fact, for $0 \leq t < t + \varepsilon < T$, we have
  \begin{align*}
    |F(t+\varepsilon)-F(t)| &\leq
    \Big| \int_0^t [e^{(t+\varepsilon-s)A}g(s) - e^{(t-s)A}g(s)]\,ds \Big|
    + \Big| \int_t^{t+\varepsilon} e^{(t+\varepsilon-s)A}g(s)\,ds \Big|\\
    &\leq |e^{\varepsilon A}-I| M_T \int_0^t |g(s)|\,ds
     + M_T \int_t^{t+\varepsilon} |g(s)|\,ds,
  \end{align*}
  and both terms converge to zero as $\varepsilon \to 0$ by definition
  of strongly continuous semigroup and because $g \in
  L_1([0,T],E)$. The case $0 < t-\varepsilon < t \leq T$ is completely
  similar, hence omitted.
\end{proof}

\bibliographystyle{amsplain}
\bibliography{ref}

\providecommand{\bysame}{\leavevmode\hbox to3em{\hrulefill}\thinspace}
\providecommand{\MR}{\relax\ifhmode\unskip\space\fi MR }
\providecommand{\MRhref}[2]{%
  \href{http://www.ams.org/mathscinet-getitem?mr=#1}{#2}
}
\providecommand{\href}[2]{#2}
\begin{thebibliography}{10}

\bibitem{agmon}
S.~Agmon, \emph{Lectures on elliptic boundary value problems}, D. Van Nostrand
  Co., Princeton, 1965. \MR{31 \#2504}

\bibitem{barbu}
V.~Barbu, \emph{Analysis and control of nonlinear infinite-dimensional
  systems}, Academic Press Inc., Boston, MA, 1993. \MR{MR1195128 (93j:49002)}

\bibitem{cm:IDAQP09}
V.~Barbu and C.~Marinelli, \emph{Strong solutions for stochastic porous media
  equations with jumps}, Infin. Dimens. Anal. Quantum Probab. Relat. Top.
  \textbf{12} (2009), no.~3, 413--426. \MR{2572464}

\bibitem{BenBre}
P.~Benilan and H.~Br{\'e}zis, \emph{Solutions faibles d'\'equations
  d'\'evolution dans les espaces de {H}ilbert}, Ann. Inst. Fourier (Grenoble)
  \textbf{22} (1972), no.~2, 311--329. \MR{MR0336471 (49 \#1245)}

\bibitem{BGJ}
K.~Bichteler, J.-B. Gravereaux, and J.~Jacod, \emph{Malliavin calculus for
  processes with jumps}, Gordon and Breach Science Publishers, New York, 1987.
  \MR{MR1008471 (90h:60056)}

\bibitem{BDMKR}
T.~Bj{\"o}rk, G.~Di~Masi, Yu. Kabanov, and W.~Runggaldier, \emph{Towards a
  general theory of bond markets}, Finance Stochast. \textbf{1} (1997),
  141--174.

\bibitem{Bmax}
H.~Br{\'e}zis, \emph{Op\'erateurs maximaux monotones et semi-groupes de
  contractions dans les espaces de {H}ilbert}, North-Holland Publishing Co.,
  Amsterdam, 1973. \MR{MR0348562 (50 \#1060)}

\bibitem{BroDinc87}
J.~K. Brooks and N.~Dinculeanu, \emph{Projections and regularity of abstract
  processes}, Stochastic Anal. Appl. \textbf{5} (1987), no.~1, 17--25.
  \MR{MR882695 (88g:60114)}

\bibitem{cerrai-libro}
S.~Cerrai, \emph{Second order {PDE}'s in finite and infinite dimension},
  Lecture Notes in Mathematics, vol. 1762, Springer-Verlag, Berlin, 2001.
  \MR{2002j:35327}

\bibitem{DP-K}
G.~Da~Prato, \emph{Kolmogorov equations for stochastic {PDE}s}, Birkh\"auser
  Verlag, Basel, 2004. \MR{MR2111320 (2005m:60002)}

\bibitem{DPR-sing}
G.~Da~Prato and M.~R{\"o}ckner, \emph{Singular dissipative stochastic equations
  in {H}ilbert spaces}, Probab. Theory Related Fields \textbf{124} (2002),
  no.~2, 261--303. \MR{MR1936019 (2003k:60151)}

\bibitem{DZ92}
G.~Da~Prato and J.~Zabczyk, \emph{Stochastic equations in infinite dimensions},
  Cambridge University Press, Cambridge, 1992. \MR{MR1207136 (95g:60073)}

\bibitem{DZ96}
\bysame, \emph{Ergodicity for infinite-dimensional systems}, Cambridge
  University Press, Cambridge, 1996. \MR{MR1417491 (97k:60165)}

\bibitem{Ebe}
A.~Eberle, \emph{Uniqueness and non-uniqueness of semigroups generated by
  singular diffusion operators}, Lecture Notes in Mathematics, vol. 1718,
  Springer-Verlag, Berlin, 1999. \MR{MR1734956 (2001c:60122)}

\bibitem{Fendler}
G.~Fendler, \emph{Dilations of one parameter semigroups of positive
  contractions on {$L\sp p$} spaces}, Canad. J. Math. \textbf{49} (1997),
  no.~4, 736--748. \MR{MR1471054 (98i:47035)}

\bibitem{Gyo-semimg}
I.~Gy{\"o}ngy, \emph{On stochastic equations with respect to semimartingales.
  {III}}, Stochastics \textbf{7} (1982), no.~4, 231--254.

\bibitem{HauSei}
E.~Hausenblas and J.~Seidler, \emph{A note on maximal inequality for stochastic
  convolutions}, Czechoslovak Math. J. \textbf{51(126)} (2001), no.~4,
  785--790. \MR{MR1864042 (2002j:60092)}

\bibitem{Jacob3}
N.~Jacob, \emph{Pseudo differential operators and {M}arkov processes. {V}ol.
  {III}}, Imperial College Press, London, 2005. \MR{MR2158336 (2006i:60001)}

\bibitem{JacShi}
J.~Jacod and A.~N. Shiryaev, \emph{Limit theorems for stochastic processes},
  second ed., Springer-Verlag, Berlin, 2003. \MR{MR1943877 (2003j:60001)}

\bibitem{Kote-Doob}
P.~Kotelenez, \emph{A stopped {D}oob inequality for stochastic convolution
  integrals and stochastic evolution equations}, Stochastic Anal. Appl.
  \textbf{2} (1984), no.~3, 245--265. \MR{MR757338 (86k:60096)}

\bibitem{KR-spde}
N.~V. Krylov and B.~L. Rozovski{\u\i}, \emph{Stochastic evolution equations},
  Current problems in mathematics, Vol. 14 (Russian), Akad. Nauk SSSR,
  Vsesoyuz. Inst. Nauchn. i Tekhn. Informatsii, Moscow, 1979, pp.~71--147, 256.
  \MR{MR570795 (81m:60116)}

\bibitem{lebedev}
V.~A. Lebedev, \emph{Fubini's theorem for parameter-dependent stochastic
  integrals with respect to {$L\sp 0$}-valued random measures}, Teor.
  Veroyatnost. i Primenen. \textbf{40} (1995), no.~2, 313--323. \MR{MR1346469
  (96g:60068)}

\bibitem{LR-heat}
P.~Lescot and M.~R{\"o}ckner, \emph{Perturbations of generalized {M}ehler
  semigroups and applications to stochastic heat equations with {L\'e}vy noise
  and singular drift}, Potential Anal. \textbf{20} (2004), no.~4, 317--344.

\bibitem{LiLo}
E.~H. Lieb and M.~Loss, \emph{Analysis}, second ed., American Mathematical
  Society, Providence, RI, 2001. \MR{MR1817225 (2001i:00001)}

\bibitem{MR}
Z.~M. Ma and M.~R{\"o}ckner, \emph{Introduction to the theory of (nonsymmetric)
  {D}irichlet forms}, Springer-Verlag, Berlin, 1992. \MR{MR1214375 (94d:60119)}

\bibitem{cm:MF10}
C.~Marinelli, \emph{Local well-posedness of {M}usiela's {SPDE} with {L}\'evy
  noise}, Math. Finance \textbf{20} (2010), no.~3, 341--363. \MR{2667893}

\bibitem{cm:JFA10}
C.~Marinelli, C.~Pr{\'e}v{\^o}t, and M.~R{\"o}ckner, \emph{Regular dependence
  on initial data for stochastic evolution equations with multiplicative
  {P}oisson noise}, J. Funct. Anal. \textbf{258} (2010), no.~2, 616--649.
  \MR{MR2557949}

\bibitem{cm:IDAQP10}
C.~Marinelli and M.~R\"ockner, \emph{On uniqueness of mild solutions for
  dissipative stochastic evolution equations}, Infin. Dimens. Anal. Quantum
  Probab. Relat. Top. (in press), arXiv:1001.5413.

\bibitem{Met}
M.~M{\'e}tivier, \emph{Semimartingales}, Walter de Gruyter \& Co., Berlin,
  1982. \MR{MR688144 (84i:60002)}

\bibitem{Pazy}
A.~Pazy, \emph{Semigroups of linear operators and applications to partial
  differential equations}, Springer-Verlag, New York, 1983. \MR{85g:47061}

\bibitem{PZ-libro}
Sz. Peszat and J.~Zabczyk, \emph{Stochastic partial differential equations with
  {L}\'evy noise}, Cambridge University Press, Cambridge, 2007. \MR{MR2356959}

\bibitem{PreRoeck}
C.~Pr{\'e}v{\^o}t and M.~R{\"o}ckner, \emph{A concise course on stochastic
  partial differential equations}, Lecture Notes in Mathematics, vol. 1905,
  Springer, Berlin, 2007. \MR{MR2329435}

\end{thebibliography}

\end{document}